\documentclass{amsart}
\usepackage{a4wide}
\usepackage{epsfig,amsmath,amsfonts,amssymb}
\usepackage{color}
\usepackage{bm}

\theoremstyle{plain}             
\newtheorem{theorem}{Theorem}[section]
\newtheorem{lem}[theorem]{Lemma}
\newtheorem{cor}[theorem]{Corollary}

\newtheorem{rem}[theorem]{Remark}
\def\theequation{\arabic{section}.\arabic{equation}}
\makeatletter\@addtoreset{equation}{section}\makeatother
\def\thetheorem{\arabic{section}.\arabic{theorem}}

\renewcommand{\text}[1]{\;\mbox{\rm #1}\;}
\newcommand{\nc}{\newcommand}
\renewcommand{\intertext}[1]{\noalign{{#1}}}
\nc{\intli}{\int\limits}
\nc{\sumli}{\sum\limits}
\nc{\cd}{\cdot}
\nc{\sbs}{\subset}
\nc{\p}{^\prime}
\nc{\pp}{^{\prime\prime}}
\nc{\wt}{\widetilde}
\nc{\wh}{\widehat}
\nc{\ov}{\overline}
\nc{\pa}{\partial}
\nc{\strich}{\,|\,}
\nc{\komma}{\,,\, }
\nc{\secol}{\,;\, }
\nc{\dopu}{\,:\,}
\nc{\f}{\frac}
\nc{\tf}{\textstyle\frac}
\nc{\ha}{\frac{1}{2}}
\nc{\tha}{\textstyle{\frac{1}{2}}}
\nc{\sm}{\setminus}
\nc{\na}{\nabla}
\nc{\8}{\infty}
\nc{\pf}{{\; \rm p.f.}\,}
\nc{\pv}{{\; \rm p.v.}\,}
\nc{\supp}{{\; \rm supp}\,}
\nc{\Gr}{\mathop{\rm{\mathbf{G}}rad}\,}
\nc{\tb}{\textbf}
\nc{\x}{\times}
\nc{\tp}{^\top}
\nc{\no}{\notag}
\nc{\os}{\overset}

\nc{\al}{\alpha}
\nc{\bet}{\beta}
\nc{\vep}{\varepsilon}
\nc{\gam}{\gamma}
\nc{\de}{\delta}
\nc{\ka}{\kappa}
\nc{\lam}{\lambda}
\nc{\om}{\omega}
\nc{\vfi}{\varphi}
\nc{\vr}{\varrho}
\nc{\si}{\sigma}
\nc{\Gam}{\Gamma}
\nc{\Om}{\Omega}
\nc{\Si}{\Sigma}

\nc{\bal}{\mathbf{\alpha}}
\nc{\ba}{\mathbf{a}}
\nc{\bc}{\mathbf{c}}
\nc{\be}{\mathbf{e}}
\nc{\bof}{\mathbf{f}}
\nc{\bg}{\mathbf{g}}
\nc{\bh}{\mathbf{h}}
\nc{\bn}{\mathbf{n}}
\nc{\bu}{\mathbf{u}}
\nc{\bv}{\mathbf{v}}
\nc{\bw}{\mathbf{w}}
\nc{\bx}{\mathbf{x}}
\nc{\by}{\mathbf{y}}
\nc{\bB}{\mathbf{B}}
\nc{\bC}{\mathbf{C}}
\nc{\bD}{\mathbf{D}}
\nc{\bF}{\mathbf{F}}
\nc{\bH}{\mathbf{H}}
\nc{\bJ}{\mathbf{J}}
\nc{\bL}{\mathbf{L}}
\nc{\bP}{\mathbf{P}}
\nc{\bW}{\mathbf{W}}

 \nc{\bnull}{\boldsymbol{0}}

\nc{\blam}{\boldsymbol{\lambda}}
\nc{\bmu}{\boldsymbol{\mu}}
\nc{\bnu}{\boldsymbol{\nu}}
\nc{\bvep}{\boldsymbol{\vep}}
\nc{\bvfi}{\boldsymbol{\vfi}}
 \nc{\bxi}{\boldsymbol{\xi}}
\nc{\bsi}{\boldsymbol{\sigma}}
\nc{\bSi}{\boldsymbol{\Sigma}}
\nc{\bPsi}{\boldsymbol{\Psi}}
\nc{\bpsi}{\boldsymbol{\psi}}
 \nc{\bTheta}{\boldsymbol{\Theta}}
\nc{\bom}{\boldsymbol\omega}

 \nc{\bcM}{\boldsymbol{\mathcal M}}

\def\cC{{\mathcal C}}
\def\cO{{\mathcal O}}
\def\cE{{\mathcal E}}
\def\cH{{\mathcal H}}
\nc{\bcH}{\mathbf{\cH}}
\def\cK{{\mathcal K}}

\def\cO{{\mathcal O}}

\def\cU{{\mathcal U}}
\def\cX{{\mathcal X}}

\def\fA{{\mathfrak A}}
\def\fM{{\mathfrak M}}

\def\d{\operatorname{d}\!}
\nc{\kreuz}{\,\Big/\kern-2.75ex\d\Theta^j\kern-4ex\Big\backslash\;\,}
\nc{\kreuzx}{\,\Big/\kern-2.5ex\d x^j\kern-3.75ex\Big\backslash\;\,}

\nc{\bbG}{\mathbb{G}}
\nc{\N}{\mathbb{N}}
\nc{\R}{\mathbb{R}}
\nc{\bbS}{\mathbb{S}}
\nc{\V}{\mathbb{V}}
\nc{\dip}{\displaystyle}
\nc{\para}{|\!|\!|}
\nc{\bigpara}{\big|\!\big|\!\big|}
\nc{\J}{\overset{\circ}{\bJ}}

\begin{document}

\title[Minimal energy problems for strongly singular Riesz kernels]
{Minimal energy problems for strongly singular Riesz kernels}

\author{Helmut Harbrecht}
\address{Helmut Harbrecht, Departement Mathematik und Informatik, Universit\"at Basel, Spiegelgasse 1,
  4051 Basel, Switzerland}
\author{Wolfgang L.~Wendland}
\address{Wolfgang L.~Wendland, Institut f\"ur Angewandte Analysis und
  Numerische Simulation, Universit\"at Stuttgart, Pfaffenwaldring~57,
  70569 Stuttgart, Germany}
\author{Natalia Zorii}
\address{Natalia Zorii, Institute of Mathematics, National Academy of
   Sciences of Ukraine, Tereshchenkivska~3, 01601, Kyiv-4, Ukraine}

\begin{abstract}
We study minimal energy problems for strongly singular
Riesz kernels $|\bx-\by|^{\alpha-n}$, where $n \ge 2$ and
$\alpha\in (-1,1)$, considered for compact $(n-1)$-dimensional
$C^\8$-manifolds $\Gam$ immersed into $\R^n$. Based on the
spatial energy of harmonic double layer potentials, we are
motivated to formulate the natural regularization of such 
minimization problems by switching to Hadamard's partie finie
integral operator which defines a strongly elliptic pseudodifferential
operator of order $\beta = 1-\alpha$ on $\Gamma$. The measures
with finite energy are shown to be 
elements from the Sobolev space
$H^{\beta/2}(\Gamma)$, $0<\beta<2$, and the corresponding
minimal energy problem admits a unique solution. We relate
our continuous approach also to the 
discrete one, which has been worked out earlier by D.P.~Hardin and E.B.~Saff.
\end{abstract}

\maketitle

\section{Introduction}\label{s:intro}
The classical Gauss problem of minimizing the Coulomb energy
to solve the problem of Thomson  and its generalization to Riesz
potentials together with the discretization is the basic problem of
many applications (in \cite{B} are listed coding theory, cubature
formulas, tight frames and packing problems). In the works \cite{B2},
\cite{B}, \cite{dr}, \cite{h-s}, and \cite{h-s2}, the discretization is
obtained by approximating the minimizing charges by a distribution
of finitely many Dirac measures on the given manifold.

If the number of Dirac points tends to infinity, then the minimizing
densities approach distributions in the form of Sobolev space elements.
Therefore, in \cite{HWZ1}, \cite{HWZ2}, \cite{OWZ}, the minimizing
measures are considered as distributions in Hilbert spaces of finite
Riesz energy. This continuous setting is simpler and more efficient
from the numerical point compared to the discrete approach in
\cite{B2}, \cite{B}, \cite{dr}, \cite{h-s}.

For potentials with Riesz kernel $|\bx-\by|^{\al-n}$, where $1<\al<n$, and
Borel measures supported on a given $(n-1)$-dimensional manifold $\Gam$
immersed into $\R^n$, a surface potential is generated, which on $\Gam$
defines a boundary integral operator with weakly singular kernel.
This boundary integral operator is a pseudodifferential
operator of \emph{negative\/} order $\beta=1-\al$ if $\Gam\in C^\8$.
The energy space of this pseudodifferential operator on $\Gam$ is
thus the Sobolev space $H^{\beta/2}(\Gam)$ of distributions and the
minimizing measure of finite energy is an element of this Sobolev space.
Hence, the determination of the minimizer is reduced
to an optimization problem with a quadratic functional which is defined in
terms of the single layer Riesz potential on $\Gam$. The strong ellipticity
of the corresponding pseudodifferential operator in $\mathbb{R}^n$ and its
trace on $\Gam$ then provides the coerciveness of the associated quadratic
functional. For $\al=2$, which corresponds to the Newtonian kernel, the
Riesz energy of the single layer potential is just its Dirichlet integral over
$\R^n\sm\Gam$.

In this paper, however, we consider the Riesz kernels with $\al\in
(-1,1)$. For $\al=0$, in classical potential theory, the energy of the
harmonic double layer potential in $\R^n\sm\Gam$ now equals the
Riesz energy if we define the latter as to be Hadamard's partie finie
integral of the hypersingular potential --- which is the natural distributional
regularization (see Section \ref{s:double} where $\Gam$ is a $(n-1)$-dimensional
planar bounded domain in $\R^n$).

Let $\Gam=\bigcup_{i\in I}\Gamma_i$ where $\Gamma_i$, $i\in I$, are
finitely many compact, connected $(n-1)$-dimensional $C^\8$-manifolds
immersed into $\R^n$. In Section \ref{s:manifold}, we then consider the Riesz
potential as a pseudodifferential operator just on $\Gamma$ since we cannot
use its extension to $\mathbb R^n$ (for $\alpha\ne0$, the transmission
conditions \cite[Theorem~8.3.11]{H-W} are not satisfied). We call the
bilinear form with the strongly singular partie finie integral of the Riesz
kernel the {\em energy of the Riesz potential\/}. The partie finie integral
operator with the hypersingular Riesz kernel defines now a strongly
elliptic pseudodifferential operator $V_\beta$ of \emph{positive\/}
order $\bet=1-\al$ on $\Gam$.

In contrast to the analysis of weakly singular Riesz kernels
provided earlier by the authors in~\cite{HWZ1}, \cite{HWZ2}, in the
case under consideration, the trace theorem in $H^{-\beta/2}(\Gamma)
= V_\beta H^{\beta/2}(\Gamma)$ is not valid anymore, because of the
negativity of the order $-\beta/2$, cf.~\cite{Adams}. Nevertheless, we have
succeeded in overcoming this difficulty, and we have shown that all the
Borel measures on $\Gamma$ with finite Riesz energy whose restriction
on any $\Gamma_i$ takes sign either $+1$ or $-1$ form a certain cone
in the Sobolev space $H^{\bet/2}(\Gam)$, $0<\bet<2$. \emph{This is our
main result in Section\/ {\rm\ref{s:manifold}}, Theorems\/ {\rm\ref{th:3.3}}
and\/ {\rm\ref{th:3.4}}}. In this framework, the corresponding Gauss variational
problem admits a unique solution which belongs to $H^{\bet/2}(\Gam)$,
which is a compact subspace of $L_2(\Gam)$. These results have again
a potential theoretic meaning in the particular situation $\alpha = 0$ in
relation to the harmonic double layer potential as explained in
Section~\ref{s:null}.

In the fundamental work \cite{h-s2} by D.P.\ Hardin and E.B.\ Saff,
discrete minimal energy problems have been investigated. There,
the discrete Riesz energies are obtained by distributing a finite
number ($N$) of evenly weighted Dirac measures on a compact
$(n-1)$-dimensional manifold $A$ where the set ${\bf x}={\bf y}$
is excluded. Then, the discrete minimal Riesz energy determines
an optimal geometric arrangement of the $N$ distinct Dirac points
on $A$. In \cite{h-s2}, three cases are distinguished: (i) the Riesz
kernel is weakly singular, (ii) the case $\alpha =1$ (see \cite{KS}),
and (iii) the hypersingular case $\alpha <1$. For all these three cases,
the behavior of the discrete minimal energies for $N$ tending to
infinity is explicitly determined (see Section~\ref{s:mini} below for details). In the works \cite{B2}, \cite{B},
and \cite{BHS}, these results are generalized to more general
Riesz kernels with weights.

During a miniworkshop in August 2012 in Stuttgart with E.B.\ Saff,
D.P.\ Hardin, and P.D.\ Dragnev, we have learned from them that in the hypersingular
case the discretized minimal energies tend to infinity if the number of
Dirac basic points approaches infinity and at the same time
those minimizing charges tend to a charge with a constant density.
This discussion inspired us to pick up this topic gratefully in our
paper and to analyze also this approach by cutting out the set
$|\bx-\by|\le\de$ of $\Gam\x\Gam$ where $\delta >0$. We first
figure out the idea in Section~\ref{s:mini} by studying a perturbed
Riesz energy problem. Then, in Section  \ref{s:comp}, we perform the
computations in detail for the punched Riesz energy problem and give
an asymptotic expansion of the solution in the corresponding family of
finite energy spaces for $\de\to 0$. In particular cases (see Corollary
\ref{c:4.6} for details), the minimizers tend to a constant distribution
on $\Gam$ while the corresponding minimal energies tend to infinity.

\section{Motivation. The energy of the Laplacian's  double layer
  potential}\label{s:double}
\setcounter{equation}{0}
\setcounter{theorem}{0}

We shall motivate our approach by an example from potential
theory where $\alpha=0$, i.e.\ $\beta = 1$. To this end, let
$\Gam\sbs\R^{n-1}$ be a planar bounded domain in $\R^n$
and $\bx=(\bx\p,x_n)\in\R^n$ with $\bx\in\Gam$ when $x_n=0$,
see Figure~\ref{fig:plane} for an illustration.

\begin{figure}
\begin{center}
\includegraphics[height=4cm,width=8cm]{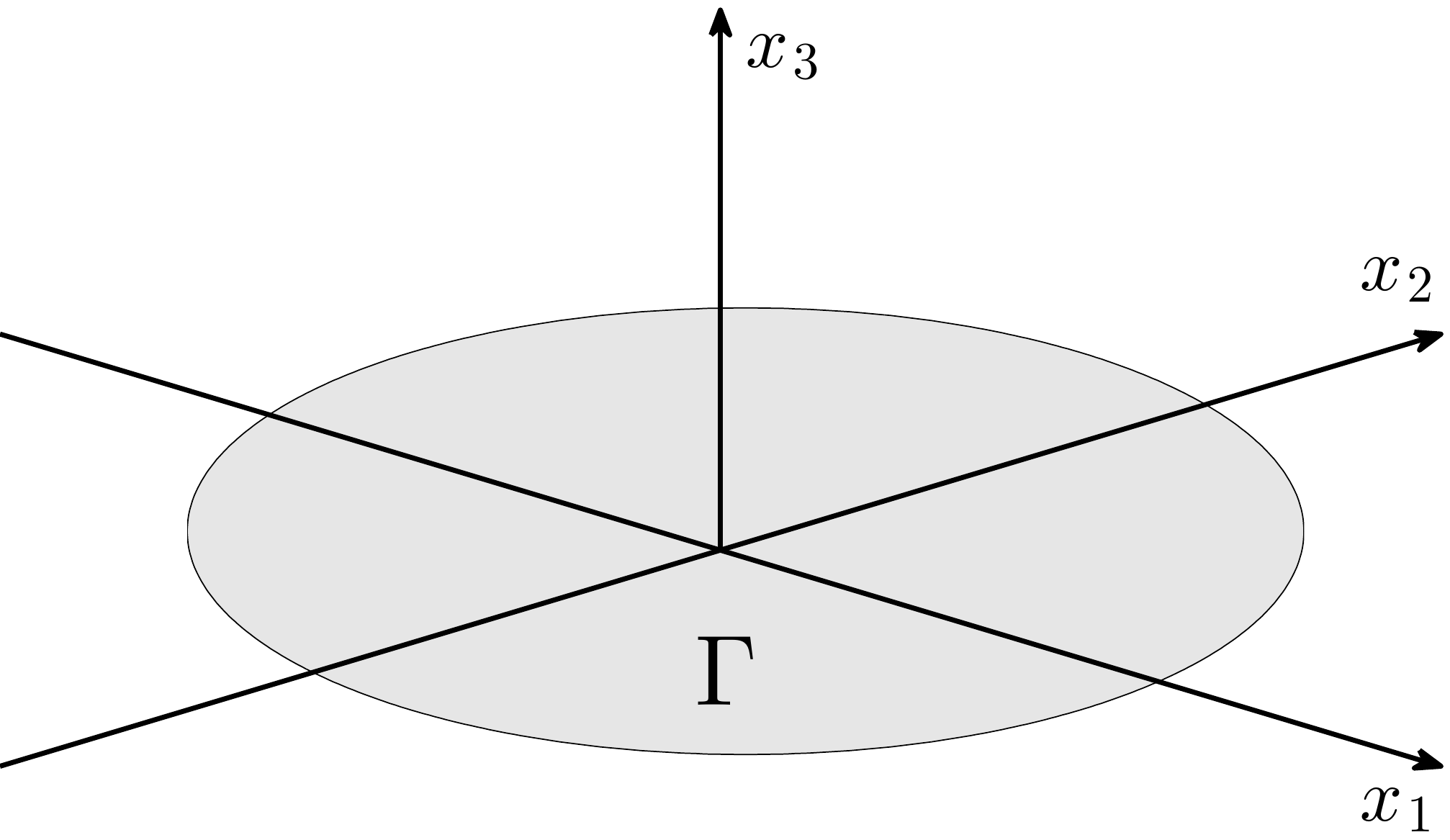}
\caption{\label{fig:plane}Illustration of the geometrical setting.}
\end{center}
\end{figure}

The double layer potential of the Laplacian with given dipole charge
density $\bvfi(\by\p)$ is given by
\begin{align*}
U_{\bvfi}(\bx):= -(\bW\bvfi)(\bx)
&:=
-\intli_{y_n=0}\bvfi(\by\p)\f{\pa E}{\pa y_n}\big((\by\p,0),\bx\big)\d\by\p\\
&\,=
- c_n\intli_{\by\p\in\Gam}\f{\big(\bx-(\by\p,0)\big)\cd
  \be_n}{|\bx-(\by\p,0)|^n}
\bvfi(\by\p)\d\by\p
\end{align*}
for $\bx\in\R^n$ with $x_n\neq 0$ and
$c_n=\f{1}{\om_n}=\big(2(n-1)\pi\big)^{-1}$.
The fundamental solution for the Laplacian is given by
\[
E(\bx,\by)=c_n|\bx-\by|^{2-n}.
\]
The vector
 $\be_n=(0,\dots,1)\tp$ is the
$n$-th basis vector of $\R^n$ and the unit normal vector on $\Gam$.
If $\bvfi$ is continuous at $\bx\p$,
then there holds the jump relation
\begin{align*}
\R_\pm^n\ni \bx
&\to
(\bx\p,0)\dopu\\
U_{\bvfi}(\bx)&\to
 \mp \ha\vfi(\bx\p)-
\intli_{\by\p\in\Gam\sm\{\bx\p\}}\f{\big((\bx-(\by\p,0)\big)\cd
  \be_n}{|\bx-(\by\p,0)|^n}
  \bvfi(\by\p)\d\by\p
=
\mp\ha \bvfi(\bx\p)-0
\end{align*}
since for $\bx\in\Gam$ and $\by\in\Gam\sm\{\bx\}$ the scalar product
$(\bx-\by)\cd \be_n=0$ and, hence, the integral vanishes.
Consequently, the harmonic potential $U_{\bvfi}(\bx)$ solves the
transmission problem in $\R^n\sm\Gam$
\[
[U]_\Gam :=U_{\bvfi}(\bx\p,-0)-U_{\bvfi}(\bx\p,+0)=\bvfi(\bx\p)
\]
where $\bvfi$ is a given element of $\widetilde{H}^{1/2}(\Gamma)$,
the closure of $C_0^\infty(\Gamma)$ in $H^{1/2}(\mathbb R^{n-1})$.

The energy of the harmonic field $U_{\bvfi}$ is given by its
Dirichlet integral, and Green's theorem yields
\begin{align*}
\lefteqn{
\intli_{\R^n\sm \Gam} |\na U_{\bvfi}(\bx)|^2\d\bx}\\
&=
-\intli_{\R^{n-1}}U_{\bvfi}(\bx\p,+0)\Big(\f{\pa}{\pa x_n}
U_{\bvfi}(\bx\p,+0)\Big)\d\bx\p+
\intli_{\R^{n-1}} U_{\bvfi}(\bx\p,-0)\Big(\f{\pa}{\pa x_n}
U_{\bvfi}(\bx\p,-0)\Big)\d\bx\p\\
&=
\intli_\Gam \ha \bvfi(\bx\p)\Big(\f{\pa}{\pa
  x_n}(-\bW\vfi)(\bx\p,-0)\Big) \d\bx\p
-\intli_\Gam \ha \bvfi(\bx\p)\Big(\f{\pa}{\pa
  x_n}(-\bW\vfi)(\bx\p,+0)\Big) \d\bx\p\\
&=
-c_n\intli_\Gam \bigg\{\ha\bvfi(\bx\p)\f{\pa}{\pa x_n}
\bigg(\intli_{\R^{n-1}}
\f{(\bx\p-x_n\be_n-\by\p)\cd\be_n}{|\bx\p-x_n\be_n-\by\p|^n}
\big(\bvfi(\by\p)-\bvfi(\bx\p)\big)\d\by\p\bigg)\\
&\qquad\qquad\
+\ha\bvfi(\bx\p)^2\f{\pa}{\pa x_n}
\bigg(\intli_{\R^{n-1}}
\f{(\bx\p-x_n\be_n-\by\p)\cd\be_n}{|\bx\p-x_n\be_n-\by\p|^n}\d\by\p\bigg)
\bigg\}\d\bx\p\\
&\qquad
+c_n\bigg\{\ha\bvfi(\bx\p)\f{\pa}{\pa x_n}
\bigg(\intli_{\R^{n-1}}
\f{(\bx\p+x_n\be_n-\by\p)\cd\be_n}{|x\p+x_n\be_n-\by\p|^n}
\big(\bvfi(\by\p)-\bvfi(\bx\p)\big)\d\by\p\bigg)\\
&\qquad\qquad
+\ha\bvfi(\bx\p)^2\f{\pa}{\pa x_n}
\bigg(\intli_{\R^{n-1}}
\f{(\bx\p+x_n\be_n-\by\p)\cd\be_n}{|\bx\p+x_n\be_n-\by\p|^n}\d\by\p\bigg)
\bigg\}\d\bx\p.
\end{align*}
We can interchange differentiation and integration in
this expression by means of Hadamard's finite part
integral. Namely, due to
\[
\intli_{\R^{n-1}} \f{(\bx\p\mp
  x_n\be_n-\by\p)\cd\be_n}{|\bx\p\mp x_n\be_n-\by\p|^n}
\d\by\p=\pm\ha\om_n
\]
with the constant $\om_n=\f{1}{c_n}$, it holds
\begin{align*}
\intli_{\R^n\sm\Gam}|\na U_{\bvfi}(\bx)|^2 \d\bx
&=
(n-2) c_n\intli_\Gam \bvfi(\bx\p)
\intli_\Gam |\bx\p-\by\p|^{-n}\big(\bvfi(\by\p)-\bvfi(\bx\p)\big) \d\by\p \d\bx\p\\
&=
(n-2)\intli_\Gam \bvfi(\bx\p)(\bD\bvfi)(\bx\p)\d\bx\p ,
\end{align*}
where
\[
\bD\bvfi(\bx\p)=-\pf c_n \intli_\Gam
\Big(\f{\pa}{\pa x_n}\
\f{\pa}{\pa y_n}|\bx-\by|^{2-n}\Big)\bvfi (\by\p)\d\by\p
\]
is Hadamard's finite part integral with $\bx=(\bx\p,0)$, $\by=(\by\p,0)$
and with the hypersingular kernel function
$k_{\bf D}({\bf x},{\bf y}) = (n-2)|\bx\p-\by\p|^{-n}c_n$,
that is
\begin{equation}\label{eq:PF}
\intli_{\R^n\sm\Gam} |\na U_{\bvfi}(\bx)|^2\d\bx
=
(n-2)c_n \intli_\Gam \pf \intli_\Gam |\bx\p-\by\p|^{-n}
\bvfi(\bx\p)\bvfi(\by\p)\d\bx\p \d\by\p.
\end{equation}
(For the definition of Hadamard's partie finie integral operators, see
 \cite{ha} and \cite[Chapter~3.2]{H-W}.)

Hence, the finite part integral on the right of \eqref{eq:PF} which has the Riesz
kernel $|\bx\p-\by\p|^{-n}$ for $\bx\p, \by\p\in\Gam\sbs \R^{n-1}$ defines
the energy of the harmonic double layer potential in $\R^n\sm\Gam$
given by the Dirichlet integral on the left of \eqref{eq:PF}. Since the Riesz
kernel is a homogeneous function of degree $-n$, it defines on $\Gam\sbs
\R^{n-1}$ a strongly elliptic pseudodifferential operator ${\bf D}$ of
order $1$ (see \cite[Section 7.1.2]{H-W}).

\section{Strongly singular Riesz energy on a
  mani\-fold}\label{s:manifold}
\setcounter{equation}{0}
\setcounter{theorem}{0}
In all that follows, without stated otherwise, we fix $n\geq2$ and $-1<\al<1$,
and write $\beta:=1-\alpha$.

In $\mathbb R^n$, consider a strongly singular Riesz
kernel $|\mathbf x-\mathbf y|^{\alpha-n}$ and a manifold $\Gam:=\bigcup_{\ell\in I}
\,\Gamma_\ell$, where $\Gamma_\ell$ are finitely many compact, connected,
mutually disjoint, boundaryless, $(n-1)$-dimensional, oriented $C^\infty$-manifolds,
immersed into~$\mathbb{R}^n$. Then, the surface measure $\d s$ on $\Gamma$
is well defined.

In what follows, $(\boldsymbol\psi,\boldsymbol\varphi)_{L^2(\Gamma)}$ will
stand for the extension of the $L^2$-scalar product to dualities as $\boldsymbol\psi
\in H^{-\bet/2}(\Gam)$ and $\boldsymbol\varphi\in H^{\bet/2}(\Gam)$ and also
to the applications of distributions $\boldsymbol\psi$ on $\Gam$ operating
on $\boldsymbol\varphi\in C^\infty(\Gam)$.

We call the strongly singular partie finie integral of the Riesz kernel
\begin{equation}\label{eq:2.1}
(V_\bet\bvfi,\bvfi)_{L_2(\Gam)}= \intli_\Gam\pf \intli_\Gam|\bx-\by|^{\al-n}
\bvfi(\bx)\bvfi(\by)\, \d s_\bx\,\d s_\by=:E_\alpha(\boldsymbol\varphi)
\end{equation}
with respect to $|\bx-\by|\ge \vep_0\to 0$, $\vep_0>0$,
operating on $\bvfi\in C^\infty(\Gamma)$, the \emph{energy\/}
of
the Riesz potential
\[
U_{\bvfi}(\bx)=\pf \intli_{\by\in\Gam} |\bx-\by|^{\al-n}\bvfi(\by)\,\d s_\by,
\quad{\bf x}\in\mathbb{R}^n,
\]
generated by the surface charge $\bvfi$ (see e.g.\ \cite{H-W}).
For $\bvfi\in C^\8(\Gam)$, the Hadamard partie finie integral operator
\[
V_\bet\bvfi(\bx)
=\pf \intli_{{\bf y}\in\Gam}|\bx-\by|^{\al-n}\bvfi(\by)\,\d s_\by,\quad
\bx\in\Gam,
\]
which underlies \eqref{eq:2.1}, is for $0\le\al<1$ given by
\[
V_\bet\bvfi(\bx)=\pv \intli_{\by\in\Gam\wedge |\bx-\by|>0}\
|\bx-\by|^{-\bet-(n-1)} \big\{\bvfi(\by)-\bvfi(\bx)\big\}\,\d s_\by
+h(\bx)\bvfi(\bx),
\]
where
\begin{equation}\label{eq:h}
h(\bx)=\pf\lim_{\de\to 0}\intli_{\by\in\Gam\wedge
  0<\de<|\bx-\by|} \
|\bx-\by|^{-\bet-(n-1)}\,\d s_\by,
\end{equation}
and for $-1<\al<0$ by
\begin{align*}
V_\bet\bvfi(\bx)
&=
\pv\!\!\intli_{\by\in\Gam\wedge |\bx-\by|>0}\!\!
|\bx-\by|^{-\bet-(n-1)}
\big\{\bvfi(\by)-\bvfi(\bx)-(\by-\bx)\cdot \na\bvfi(\bx)\big\}\,\d s_\by\\
&\hspace*{5cm}
+h(\bx)\bvfi(\bx)+\bh(\bx)\cdot \na\bvfi(\bx),
\end{align*}
where
\begin{equation}\label{eq:bold h}
\bh(\bx)=\pf\lim_{\de\to 0}\intli_{\by\in\Gam\wedge 0<\de<|\bx-\by|}
|\bx-\by|^{-\bet-(n-1)}(\by-\bx)\,\d s_\by.
\end{equation}
The abbrevation $\pv$ stands for the
Calderon--Mikhlin principal value integral (see \cite{Mi-Pr}). (See Appendix~\ref{a:a}
for the explicit computation of the partie finie integrals $h(\bx)$ and
$\bh(\bx)$.)

\begin{theorem}[see {\cite[Chapter 8]{H-W}}]\label{th:2.1}
The partie finie integral operator $V_\bet$ is a strongly elliptic
pseudodifferential operator of order $\bet=1-\al\in (0,2)$ on $\Gam$.
The principal symbol of this operator is given by the equivalence
class associated with the homogeneous function
\begin{equation}\label{eq:sy}
 a(\bxi)=
C(n-1,\bet)|\bxi|^\bet,
\text{where}C(n-1,\bet)
=
2^{-\bet}\pi^{\f{n-1}{2}}
 \tf{\Gam(\f{-\bet}{2})}{\Gam(\f{n-\al}{2})} \text{and} \xi\in\mathbb{R}^{n-1}.
\end{equation}
$V_\bet$  defines the linear and
continuous mapping
$V_\bet\dopu H^s(\Gam)\to H^{s-\bet}(\Gam)$ for every $s\in\R$. In
particular, for
$s=\bet/2$, $V_\bet$ maps $H^{\bet/2}(\Gam)$ into $H^{-\bet/2}(\Gam) $
and there exist $0<c_0\le c_2$ and $c_1\ge 0$ such that the inequalities
\begin{equation}\label{eq:2.2}
c_0\|\bvfi\|^2_{H^{\bet/2}(\Gam)}-c_1\|\bvfi\|^2_{L^2(\Gam)}
\le (V_\bet\bvfi,\bvfi)_{L_2(\Gam)}\le c_2
\|\bvfi\|_{H^{\bet/2}(\Gam)}^2
\end{equation}
are satisfied for any $\bvfi\in H^{\bet/2}(\Gam)$.
\end{theorem}

\begin{proof}
For justifying the inequalities \eqref{eq:2.2}, recall that for each of
the components of the  $C^\infty$-manifolds $\Gamma_\ell$,
$\ell\in I$, immersed into~$\mathbb R^n$, we may associate a family
of finite-dimensional atlases $\mathfrak A_{\ell}$ (see \cite{ku}). Each
atlas $\mathfrak A_{\ell}$ is a family of local charts $(O_{{\ell}r},
\mathcal U_{{\ell}r}, \mathcal X_{{\ell}r})$, where $r$ ranges
through a finite  set~$R_{\ell}$. The open sets
$O_{{\ell}r}\subset\Gamma_{\ell}$ define an open covering of
$\Gamma_{\ell}$, while $\mathcal X_{\ell r}$ is a
$C^\infty$-diffeomorphism of $O_{{\ell}r}$ onto $\mathcal
U_{{\ell}r}\subset\mathbb R^{n-1}$. Let $\{\beta_{{\ell}r}\}_{r\in
R_{\ell}}$ be a $C^\infty$-partition of unity of $\Gamma_{\ell}$
which is subordinate to the atlas $\mathfrak A_{\ell}$. In addition
to the partition of unity, let $\{\gamma_{{\ell}r}\}_{r\in R_{\ell}}$
be a second system of functions $\gamma_{{\ell}r}\in
C_0^\infty(O_{{\ell}r})$ with the properties
\[
\gamma_{{\ell}r}({\bf x})=1\text{for all} {\bf x}\in{\rm
supp}\,\beta_{{\ell}r}\text{and} 0\leq \gamma_{\ell r}.
\]
Thus, it holds that
\[
\gamma_{{\ell}r}({\bf
x})\beta_{{\ell}r}({\bf x}) =\beta_{{\ell}r}({\bf
x})\text{and}\beta_{{\ell}r}({\bf
x})\gamma_{{\ell}r}({\bf x})=\beta_{{\ell}r}({\bf x})\text{for
all}{\bf x}\in\Gamma_{\ell}.
\]
With respect to the atlas $\mathfrak A_{\ell}$, let
$\mathcal X_{{\ell}r\star}$ denote the corresponding
pushforwards and $\mathcal X_{{\ell}r}^\star$ the pullbacks.
Then $\mathcal X_{{\ell}r\star}\beta_{{\ell}r}\in C_0^\infty
(\mathcal U_{{\ell}r})$.

Without loss of generality, the local parametric representations can
always be chosen in such a way that at one point ${\bf
x}^\circ_{{\ell}r}\in O_{{\ell}r}$ where $\beta_{{\ell}r}({\bf
x}^\circ_{{\ell}r})=1$ we have $\mathcal X_{{\ell}r}({\bf
x}^\circ_{{\ell}r})={\bf 0}$ and, moreover, at this point the tangent
bundle
\[
\frac{\partial{\bf x}}{\partial{\bf x}'}\Biggl|_{{\bf
x}'={\bf 0}}=\frac{\partial\mathcal X_{{\ell}r}^{-1}({\bf
x}')}{\partial{\bf x}'}\Biggl|_{{\bf x}'={\bf 0}},\text{where}
{\bf x}':=\mathcal X_{{\ell}r}({\bf x}),
\]
forms a positively oriented system of $n-1$ mutually orthogonal unit vectors.
This implies that the Riemannian tensor of $\Gamma_{\ell}$ in the local
coordinates at the point ${\bf x}^\circ_{{\ell}r}$ is the unity matrix.
Hence,  the surface measure satisfies
\[
\d s_{\ell}({\bf x}) =J_{{\ell}r} ({\bf x}')\,\d{\bf
x}\p\text{where}{\bf x}\p\in\mathcal U_{{\ell}r},\
\text{and} J_{\ell r}({\bf 0})={\bf 1}.
\]

Given an atlas $\fA_\ell$ on $\Gam_\ell$, define
\[
d_\ell:= \min_{r\in R_\ell}{\rm diam\ } \cU_{\ell r}.
\]
Thus, one can choose $\delta_0>0$ so that for any
given $0<\delta<\delta_0$ there exists a finite-dimensional
atlas~$\mathfrak A^\delta_{\ell}$ satisfying all the above formulated
properties
  and $d_{\ell}=\delta$. Hence, we have a
whole family of finite atlases~$\mathfrak A^\delta_{\ell}$ with
$0<\delta<\delta_0$ which will be under consideration. (We then
shall omit the index~$\delta$ in the notation.)

Note that the Jacobians~$J_{{\ell}r}$ depend on the geometric
properties of $\Gamma_{\ell}$ only, and
$J_{{\ell}r}$ together with their derivatives are uniformly
continuous relative to $\delta\in(0,\delta_0)$.

Corresponding to the partition of unity, the pseudodifferential
operator~$V_\bet$ on~$\Gamma$ can be decomposed  as
\[
V_\bet
=
\sumli_{\ell\in I}\ \sumli_{r\in R_\ell}
\bet_{\ell r} V_\bet\gam_{\ell r}+Z_1
=
\sumli_{\ell\in I}\ \sumli_{r\in R_\ell}
\gam_{\ell r}V_\bet\bet_{\ell r}+Z_2
\]
and
\begin{equation}\label{eq:l1}
V_\bet
=\sum_{{\ell}\in I}\,\sum_{r\in R_{\ell}}\,
\beta_{{\ell}r} \cX_{\ell r}^\star V_{\ell r}\cX_{\ell
  r\star}\gamma_{{\ell}r}+Z_1 =\sum_{{\ell}\in I}\,\sum_{r\in
R_{\ell}}\,\gamma_{{\ell}r} \cX_{\ell r}^\star V_{\ell r}\cX_{\ell
  r\star}\beta_{{\ell}r}+Z_2.
\end{equation}
Herein, $Z_1,Z_2$ are smoothing operators of order $-\infty$ in view
of $|\bx-\by|^{\al-n}\in C^\infty(\Gam_p\x\Gam_k)$ for $p\neq k$ and
$\supp(1-\gam_{\ell r}) \cap \supp (\bet_{\ell r})=\emptyset$ if
$\bx,\by\in\Gam_\ell$ (see also \cite[Chapter 8]{H-W}).
Moreover,
\begin{equation}\label{eq:l2}
V_{{\ell}r}\bvfi({\bf x}\p) = \pf\int_{\mathcal
U_{{\ell}r}}\,\bigl|\mathcal X_{{\ell}r}^{-1}({\bf x}\p)-\mathcal
X_{{\ell}r}^{-1}({\bf y}\p)\bigr|^{\alpha-n}\bvfi({\bf y}\p)
J_{{\ell}r}({\bf y}\p)\d{\bf y}\p,\ \bvfi\in
C_0^\8(\cU_{\ell r})
\end{equation}
are the localized operators in the parametric domains
$\mathcal U_{{\ell}r}$, defined by the operator~$V_\bet$.

The inequalities \eqref{eq:2.2} now follow locally on each chart of
the atlas $\fA_\ell$ for the localized operators $V_{\ell r}$ in
local coordinates in $\cU_{\ell r}$.
With Martensen's surface polar coordinates
(\eqref{eq:b5}, \eqref{eq:3++} in Appendix \ref{a:b}),
the kernel of $V_{\ell r}$ admits a pseudohomogeneous
asymptotic expansion of the form
\[
k(\mathbf{x}^\prime,\varrho)=\varrho^{-\beta-(n-1)}\Bigg\{1+\sum_{j\geq 2} k_{\ell r,j}
(\mathbf{x}^\prime,\varrho)\Bigg\},
\]
where
\[
 k_{\ell r,j}(\mathbf{x}^\prime,t\varrho)=
 t^j k_{\ell r,j}(\mathbf{x}^\prime,\varrho)\text{for} t>0,\ \varrho>0
\]
 since $|\mathbf{x}^\prime- \mathbf{x}_0^\prime|^2$ satisfies
 the expansion (B9). Correspondingly, the symbol $a(\mathbf{x}^\prime,
 \boldsymbol{\xi}^\prime)$ of $V_{\ell r}$ has the asymptotic expansion
\[
 a(\mathbf{x}^\prime,\boldsymbol{\xi}^\prime)=
 |\boldsymbol{\xi}^\prime|^\beta \Bigg\{1+\sum_{j\geq 2}
 a^0_{-j}(\mathbf{x}^\prime,\boldsymbol{\xi}^\prime)\Bigg\},
\]
where $a^0_{-j}(\mathbf{x}^\prime,t\boldsymbol{\xi}^\prime)=
t^{-j}a^0_{-j}(\mathbf{x}^\prime,\boldsymbol{\xi}^\prime)$
with $t>0$, $\boldsymbol{\xi}^\prime\neq {\bf 0}$.

Then, Fourier transform and Parseval's theorem yield
\begin{equation}\label{eq:3.8A}
(V_{\ell r}\boldsymbol\varphi,\boldsymbol\varphi)_{L_2(\cU_{\ell r})}
=\int\limits_{\mathbb{R}^{n-1}}|\bxi^\prime|^\beta |\widehat{\bvfi}(\bxi^\prime)|^2 \d\bxi^\prime
+(A_{\ell r}^{\beta-2}\bvfi,\bvfi)_{L_2(\mathcal{U}_{\ell r)}}
+(R_{\ell r}\bvfi,\bvfi)_{L_2(\mathcal{U}_{\ell r})},
\end{equation}
where $A_{\ell r}^{\beta-2}$ is a pseudodifferential
operator of order $\beta-2$:
\[
(A_{\ell r}^{\beta-2}\bvfi)(\mathbf{x}^\prime) =
\int\limits_{\mathbb{R}^{n-1}} e^{i\mathbf{x}^\prime\cdot\bxi^\prime}
\bPsi_{\ell r}(\bxi^\prime)|\bxi^\prime|^\beta \sum\limits_{j\geq 2} a_{-j}^0
(\mathbf{x}^\prime,\bxi^\prime)\widehat{\bvfi}(\bxi^\prime) \d\bxi^\prime.
\]
Here, the function $\bPsi_{\ell r}\in C^\infty(\mathbb{R}^{n-1})$ is arbitrary
but fixed such that $0\leq\bPsi_{\ell r}(\bxi^\prime)\leq 1$ and
\[
\bPsi_{\ell r}(\bxi^\prime) =\bnull \text{for}|\bxi^\prime|\leq\ha
\text{and}\bPsi_{\ell r}(\bxi^\prime)=1 \text{for}|\bxi^\prime|\geq 1.
\]
The remainder operator
$R_{\ell r}\bvfi (\mathbf{x})= \int_\Gamma R_{\ell r} (\mathbf{x},
\mathbf{y})\bvfi (\mathbf{y})\d s_{\bf y}$ is a smoothing operator
with the smooth kernel function $R_{\ell r}\in C^\infty (\Gamma
\times\Gamma)$. Hence, there exists a constant $c_{\ell r}>0$
such that
\[
(V_{\ell r}\bvfi,\bvfi)_{L_2(\mathcal{U}_{\ell r})}
\leq
c_{\ell r} \int\limits_{\mathbb{R}^{n-1}}(1+|\boldsymbol{\xi}^\prime|)^\beta
|\widehat{\boldsymbol\varphi}(\boldsymbol{\xi}^\prime)|^2 \d\boldsymbol{\xi}^\prime
=
c \|\boldsymbol\varphi\|^2_{H^{\beta/2}(\cU_{\ell r})}
\]
which implies with some constant $c_2>0$ that
\begin{equation}\label{eq:3.11r}
(V_\beta\boldsymbol\varphi,\boldsymbol\varphi)_{L_2(\Gamma)}
\leq
c_2\|\boldsymbol\varphi\|^2_{H^{\beta/2}(\Gamma)}.
\end{equation}
Vice versa, from
\[
 (V_\beta\boldsymbol\varphi,\boldsymbol\varphi)_{L_2(\cU_{\ell r})}
\geq
\int\limits_{\mathbb{R}^{n-1}} (1+|\boldsymbol{\xi}^\prime|)^\beta |\widehat{\boldsymbol\varphi}
(\boldsymbol{\xi}^\prime|^2\d\boldsymbol{\xi}^\prime
-c^\prime_{\ell r}\int\limits_{\mathbb{R}^{n-1}} (1+|\boldsymbol{\xi}^\prime|)^{\beta-2}
|\widehat{\boldsymbol\varphi}(\boldsymbol{\xi}^\prime)|^2\d\boldsymbol{\xi}^\prime,
\]
after summation over $\ell\in I$, we obtain the G{\aa}rding inequality
\[
(V_\beta\boldsymbol\varphi,\boldsymbol\varphi)_{L_2(\Gamma)}
\geq c_0 \|\boldsymbol\varphi\|^2_{H^{\beta/2}(\Gamma)}
-c_1\|\boldsymbol\varphi\|^2_{L_2(\Gamma)}.
\]

The embedding $H^{\beta/2}(\Gamma)\hookrightarrow L_2(\Gamma)$
is compact since $0<\beta<2$. For $\beta=1\in\mathbb{N}_0$ the
Tricomi condition needs to be satisfied for $V_\beta$ being a
pseudodifferential operator and reads here as
\[
 \int\limits_{|\Theta|=1} \Theta^{\boldsymbol\alpha^\prime}\d\omega(\Theta)=0
 \text{for}|\boldsymbol\alpha^\prime|=1
\]
(see \cite[Theorem 7.1.7]{H-W}). In addition, we have
$a^0_{-1}(\mathbf{x}^\prime,\boldsymbol{\xi}^\prime)=0$. So, $V_\beta$
is a classical pseudodifferential operator on $\Gamma$ of order $\beta\in (0,2)$.
\end{proof}

Now, we introduce  a set of so-called \emph{admissible measures}
or charges located on $\Gam$. Recall that $\Gamma=\bigcup_{\ell\in I}\Gam_\ell$,
where the finitely many $\Gam_\ell$ are compact, nonintersecting, boundaryless,
connected, $(n-1)$-dimensional, orientable $C^\8$-manifolds, immersed into
$\R^n$. With each $\Gamma_\ell$ we associate a prescribed sign $\al_\ell\in\{-1,1\}$
where $\al_\ell=+1$ for $\ell\in I^+$ and $\al_\ell=-1$ for $\ell\in I^-$. Then $I=I^+\cup I^-$
and $I^+\cap I^-=\emptyset$, $I^-=\emptyset$ is admitted.
Let $\fM(\Gam)$ denote the $\si$-algebra of signed Borel measures $\bnu$ on
$\Gam$ equipped with the topology of pointwise convergence on
$C(\Gam)$, the class of all real-valued continuous functions on
$\Gam$.

Next, consider the manifold $\Gam$ being loaded by charges
of the form
\begin{equation}\label{eq:herzherz}
\bmu=\sum_{\ell\in I}\, \al_\ell\mu^\ell
\end{equation}
where, for every $\ell\in I$, $\mu^\ell$ is a nonnegative Borel measure
on $\Gamma_\ell$. The convex cone of all
signed measures $\bmu$ of the form~\eqref{eq:herzherz} will be denoted by
$\mathfrak{M}^+(\Gamma)$.

The following theorem deals with absolutely continuous
$\boldsymbol\Sigma\in\mathfrak{M}^+(\Gamma)$, i.e.\
$\mathrm d\boldsymbol\Sigma=\boldsymbol\sigma\,\d s$,
with densities $\boldsymbol\sigma\in\mathcal K^{\beta/2}(\Gamma)$,
\[
\mathcal K^{\beta/2}(\Gamma):=\Bigg\{\boldsymbol\sigma
=\sum_{\ell\in I}\,\alpha_\ell\sigma^\ell,
\ \text{where\ }\sigma^\ell\in H^{\beta/2}(\Gamma_\ell)
\text{\ and\ }\sigma^\ell\geq 0\Bigg\}.
\]
For brevity, we shall often identify an absolutely continuous
Borel measure $\boldsymbol\Sigma\in\mathfrak{M}^+(\Gamma)$
with~$\boldsymbol\sigma$, its density. Likewise, the cone of all
$\boldsymbol\Sigma\in\mathfrak{M}^+(\Gamma)$ with
$\boldsymbol\sigma\in\mathcal K^{\beta/2}(\Gamma)$ will be
denoted by~$\mathcal K^{\beta/2}(\Gamma)$, provided that
this will not cause any misunderstanding. Similar to as it has
been done in~(\ref{eq:2.1}), we define {\it the Riesz energy\/}
of $\boldsymbol\Sigma=\boldsymbol{\sigma}\in\mathcal
K^{\beta/2}(\Gamma)$ by
\[
E_\alpha(\boldsymbol{\Sigma})
:= (V_\beta\boldsymbol{\sigma},\boldsymbol{\sigma})_{L_2(\Gamma)}
= \intli_\Gamma \pf \intli_{\Gamma}|{\bf x}-{\bf y}|^{\alpha-n}\,
\d\boldsymbol{\Sigma}({\bf x})\d\boldsymbol{\Sigma}({\bf y}).
\]

\begin{theorem}\label{th:x}
For any\/ $\boldsymbol\Sigma=\boldsymbol\sigma\in
\mathcal K^{\beta/2}(\Gamma)$, the Riesz energy is finite
and satisfies the inequalities
\begin{equation}\label{eq:stp}
c_0^\prime\|\boldsymbol\sigma\|^2_{H^{\beta/2(\Gamma)}}\leq E_\alpha(\boldsymbol{\Sigma})
:= (V_\beta\boldsymbol{\sigma},\boldsymbol{\sigma})_{L_2(\Gamma)}
\leq c_1^\prime\|\boldsymbol\sigma\|^2_{H^{\beta/2(\Gamma)}},
\end{equation}
the constants\/ $c_0^\prime$ and\/ $c_1^\prime$ being
strictly positive and independent of\/~$\bSi$. This means
that $V_\beta$ is continuously invertible on $\mathcal
K^{\beta/2}(\Gamma)$.
\end{theorem}

\begin{proof}
Write
\[
\mathcal{K}^\infty(\Gam):=\Bigg\{\boldsymbol\sigma
= \sum\limits_{\ell\in I}\,\alpha_\ell \sigma^\ell,\
\text{where}\ \sigma^\ell\in C^\infty(\Gamma_\ell)\text{\ and\ }
\sigma^\ell\geq 0\text{\ on}\ \Gamma_\ell\Bigg\}.
\]
Let $\psi(\varrho)$ be a $C_0^\infty(\mathbb{R})$-function with
$\psi(\varrho)=1$ for $0\leq\varrho\leq\delta$, $0\leq\psi(\varrho)\leq 1$
and $\psi(\varrho)=0$ for $\varrho\geq 2\delta>0$. Having the atlas
$\mathfrak{A}_\ell^\delta$ at hand, we use the decomposition
\eqref{eq:l1} and the representation in local coordinates \eqref{eq:l2}
to arrive at
\begin{align*}
  (V_{\ell r}\boldsymbol\sigma)(\mathbf{x}^\prime)
  &=
  \pf \int\limits_{\mathbb{R}^{{n-1}}}|\mathbf{x}^\prime-\mathbf{y}^\prime|^{-\beta-(n-1)}
  \psi_\star(|\mathbf{x}^\prime-\mathbf{y}^\prime|)\boldsymbol\sigma_\star(\mathbf{y}^\prime) J_{\ell r}
  (\mathbf{y}^\prime)\, \d\mathbf{y}^\prime\\
  & \quad +
  \int\limits_{|\mathbf{x}^\prime-\mathbf{y}^\prime|\geq\delta} |\mathbf{x}^\prime-\mathbf{y}^\prime|^{-\beta-(n-1)}
  \big(1-\psi_\star(|\mathbf{x}^\prime-\mathbf{y}^\prime|)\big)\boldsymbol\sigma_\star(\mathbf{y}^\prime) J_{\ell r}
(\mathbf{y}^\prime)\, \d\mathbf{y}^\prime,
\end{align*}
where $\psi_\star$ and $\boldsymbol\sigma_\star$ are the pushforwards
of $\psi$ and $\boldsymbol\sigma$, respectively. Summation gives
\begin{align*}
 (\boldsymbol\sigma,V_\beta\boldsymbol\sigma)_{L_2(\Gamma)}
   &= \sum_{\ell,i\in I} \sum_{r\in R_\ell} \sum_{q\in R_i}
   \int\limits_{\mathbb{R}^{{n-1}}}\!\!\pf\! \!\int\limits_{\mathbb{R}^{{n-1}}}
  |\mathbf{x}^\prime-\mathbf{y}^\prime|^{-\beta-(n-1)}
   \psi_\star(|\mathbf{x}^\prime-\mathbf{y}^\prime|)\\
   &\hspace*{3cm}\cdot \boldsymbol\sigma_\star(\mathbf{y}^\prime)
   J_{\ell r}(\mathbf{y}^\prime)\d\mathbf{y}^\prime
   \boldsymbol\sigma_\star(\mathbf{x}^\prime) J_{iq}(\mathbf{x}^\prime)
   \d \mathbf{x}^\prime\\
   &\quad +
   \iint\limits_{|\mathbf{x}-\mathbf{y}|\geq\delta}
   |\mathbf{x}-\mathbf{y}|^{-\beta-(n-1)}
   \big(1-\psi(|\mathbf{x}-\mathbf{y}|)\big)\boldsymbol\sigma(\mathbf{y})\d s_{\bf y}
   \boldsymbol\sigma(\mathbf{x})\d s_{\bf x}.
\end{align*}

Since the localized operators $V_{\ell r}$ are all pseudodifferential
operators of order $\beta$ with positive definite principal symbol
\eqref{eq:sy}, hence strongly elliptic, one finds with Fourier transform
and Parseval's theorem the estimate
\begin{align*}
&\sum_{\ell,i\in I} \sum_{r\in R_\ell} \sum_{q\in R_i}\
\int\limits_{\mathbb{R}^{n-1}}\pf \int\limits_{\mathbb{R}^{n-1}}
|\mathbf{x}^\prime-\mathbf{y}^\prime|^{-\beta-(n-1)}
\psi_\star(|\mathbf{x}^\prime-\mathbf{y}^\prime|)\\
&\hspace*{5cm}\cdot\boldsymbol\sigma_\star(\mathbf{y}^\prime)
J_{\ell r}(\mathbf{y}^\prime)\d\mathbf{y}^\prime
\boldsymbol\sigma_\star(\mathbf{x}^\prime) J_{iq}(\mathbf{x}^\prime)\d \mathbf{x}^\prime \\
&\qquad\geq c_0^\prime\|\boldsymbol\sigma\|^2_{H^{\beta/2}(\Gamma)}
-c^{\prime\prime}\frac{1}{2-\beta}\delta^{2-\beta} \|\boldsymbol\sigma\|^2_{L_2(\Gamma)}
\geq c_0^{\prime\prime\prime}\|\boldsymbol\sigma\|^2_{H^{\beta/2}(\Gamma)}
\end{align*}
with $c_0^{\prime\prime\prime}>0$
since $H^{\beta/2}(\Gamma)\hookrightarrow L_2(\Gamma)$
compactly and if $\delta>0$ is chosen sufficiently small.
The remaining quadratic form
\[
\iint\limits_{|\mathbf{x}-\mathbf{y}|\geq\delta}\ |\mathbf{x}-\mathbf{y}|^{-\beta-(n-1)}
\big(1-\psi(|\mathbf{x}-\mathbf{y}|)\big) \boldsymbol\sigma(\mathbf{y})
\d s_{\bf y}\boldsymbol\sigma(\mathbf{x})\d s_{\bf x}
\]
has a strictly positive $C^\infty$-kernel. Hence the left inequality in
\eqref{eq:stp} is satisfied for $\boldsymbol\sigma\in\mathcal{K}^\infty(\Gam)$,
$\boldsymbol\sigma\not\equiv {\bf 0}$. Since $\mathcal{K}^\infty(\Gam)$
is dense in $\mathcal{K}^{\beta/2}(\Gam)$, the left inequality in \eqref{eq:stp}
also holds for $\boldsymbol\sigma\in\mathcal{K}^{\beta/2}(\Gam)$,
$\boldsymbol\sigma\not\equiv {\bf 0}$ by completion.
The right inequality in \eqref{eq:stp} was already shown in \eqref{eq:3.11r}
for~$\boldsymbol\sigma$ in place of~$\boldsymbol\varphi$.
This completes the proof.
\end{proof}

We next proceed by defining the notion of the Riesz energy for
arbitrary (not necessarily absolutely continuous) measures
$\boldsymbol{\Sigma}\in\mathfrak{M}^+(\Gam)$. Since $V_\beta$
is a classical pseudodifferential operator on $\Gamma$, it maps the
distribution given by the Radon measure $\bSi\in\mathfrak{M}^+(\Gamma)$
to $V_\beta\bSi$ which is a distribution again, and this linear mapping is
continuous in the weak topology of distributions (see Theorem II.1.5 in \cite{Tay}).
Therefore, the action of the measure $V_\beta\bSi$ on functions
$\bvfi\in C^\infty(\Gamma)$ is well defined and
\[
(V_\beta\bSi,\boldsymbol\varphi)_{L_2(\Gamma)}
=(\bSi,V_\beta\boldsymbol\varphi)_{L_2(\Gamma)},
\]
since $V_\beta$ is symmetric.

Let $\mathcal{E}_\alpha^+(\Gamma)$ consist of
all $\bSi\in \mathfrak{M}^+(\Gamma)$ which the property
\[
\sup_{\|\bvfi\|_{H^{\beta/2}(\Gamma)}\leq 1}
\big|(V_\beta\bSi,\bvfi)_{L_2(\Gamma)}\big|<\infty.
\]
Hence, for $\bSi\in\mathcal{E}_\alpha^+(\Gamma)$, we
can identify  $V_\beta\bSi$ with an associated element
$\bpsi \in H^{-\beta/2}(\Gamma)$ satisfying $\bpsi \d s=\d V_\beta\bSi$ and
\begin{equation*}\label{z-2}
\|\bpsi\|_{H^{-\beta/2}(\Gamma)}=\sup_{\|\bvfi\|_{H^{\beta/2}(\Gamma)}\le 1}
\big|(V_\beta\bSi,\bvfi)_{L_2(\Gamma)}\big|.
\end{equation*}
This leads us to the following theorem.

\begin{theorem}[see also {\cite[Theorem 3]{HWZ1}}]\label{th:3.3}
For any $\bSi\in\mathcal{E}_\alpha^+(\Gamma)$
there exists a unique element $\bsi\in\mathcal{K}^{\beta/2}(\Gamma)$
such that $\d\boldsymbol\Sigma=\boldsymbol\sigma\,\d s$
and
\begin{equation}\label{eq:2.s}
\bSi(\bvfi)=\intli_\Gam\bvfi\d\bSi=(\bvfi,\bsi)_{L_2(\Gam)}
\text{for all} \bvfi\in C^\8(\Gam).
\end{equation}
Moreover, $\mathcal{E}_\alpha^+(\Gamma)=\mathcal K^{\beta/2}(\Gamma)$,
and the Riesz energy\/ $E_\alpha(\boldsymbol\Sigma)$ of any
$\bSi\in\mathcal{E}_\alpha^+(\Gamma)$ is equivalent to the\/
$H^{\beta/2}(\Gamma)$-norm of the corresponding\/ $\boldsymbol
\sigma\in\mathcal K^{\beta/2}(\Gamma)$ in the sense of\/~{\rm(\ref{eq:stp})}.
\end{theorem}

\begin{proof}
Choose an arbitrary $\bSi\in\mathcal{E}_\alpha^+(\Gamma)$.
As has been observed just above, $V_\beta\bSi\in H^{-\beta/2}(\Gamma)$
is a linear functional on $C^\infty(\Gamma)$, and it is bounded on
$H^{\beta/2}(\Gamma)$ because of
\[
|V_\bet\bSi(\bvfi)|=\bigl|(V_\beta\bSi,\boldsymbol\varphi)_{L_2(\Gamma)})\bigr|\le
c\|V_\beta\bSi\|_{H^{-\beta/2}(\Gamma)}\,\|\bvfi\|_{H^{\bet/2}(\Gam)}.
\]
Since $C^\8(\Gam)$ is dense in
$H^{\bet/2}(\Gam)$, according to the Fischer--Riesz lemma
on the representation of bounded linear functionals there
exists a unique $\bsi_0\in H^{\bet/2}(\Gam)$ with
\[
\intli_\Gam V_\bet\bvfi\d\bSi=
\intli_\Gam\boldsymbol\zeta \d\bSi =\bSi (\boldsymbol\zeta)
=\intli_\Gam\bsi_0\boldsymbol\zeta\d s=(\bsi_0,\boldsymbol\zeta)_{L_2(\Gam)},
\]
where $\boldsymbol\zeta:=V_\bet\bvfi$. If $\bvfi$ traces $C^\8(\Gam)$,
so does $\boldsymbol\zeta$; hence, \eqref{eq:2.s} holds for
$\boldsymbol\sigma=\boldsymbol\sigma_0\in H^{\beta/2}(\Gamma)$
by replacing $\boldsymbol\zeta$ by $\bvfi$. Since $\bSi\in\mathfrak{M}^+
(\Gamma)$, we actually have $\bsi_0\in\mathcal{K}^{\beta/2}(\Gamma)$,
and the inclusion $\mathcal{E}_\alpha^+(\Gamma)\subseteq
\mathcal{K}^{\beta/2}(\Gamma)$ follows.

Now, let $\boldsymbol\sigma\in\mathcal{K}^{\beta/2}(\Gamma)$. Then
$\boldsymbol\psi=V_\beta\boldsymbol\sigma\in H^{-\beta/2}(\Gamma)$ and
\begin{align*}
\|\boldsymbol\psi\|_{H^{-\beta/2}(\Gamma)}&=\|V_\beta
\boldsymbol\sigma\|_{H^{-\beta/2}(\Gamma)}\\{}&=
\sup_{\|\boldsymbol\varphi\|_{H^{\beta/2}(\Gamma)}\leqslant1}\,\big|
(V_\beta\boldsymbol\sigma,\boldsymbol\varphi)_{L_2(\Gamma)}\big|
\leqslant c\|\boldsymbol\sigma\|_{H^{\beta/2}(\Gamma)}<\infty
\end{align*}
due to the duality $H^{-\beta/2}(\Gamma)\times H^{\beta/2}(\Gamma)$
of $(\cdot,\cdot)_{L_2(\Gamma)}$, $V_\beta$ being a pseudodifferential
operator on~$\Gamma$. Hence, $\mathcal{K}^{\beta/2}(\Gamma)
\subseteq\mathcal E^+_\alpha(\Gamma)$, which completes the proof.
\end{proof}

Although we have shown that the distributions in $\mathcal{K}^{\bet/2}(\Gam)$
all have finite Riesz energy $E_\alpha(\bmu)$, it is not clear yet whether there
are no other measures in $\mathfrak{M}^+(\Gamma)$ whose Riesz energy is
finite. To elaborate on this problem, we employ an idea by J.~Deny \cite{De}.
A measure on $\Gamma$ can be considered as a distribution on $\Gamma$
and, hence, can be Fourier transformed. In connection with the localization
of $V_\beta$ on one chart of the atlas on $\Gamma$, we have relation
\eqref{eq:3.8A} where the pseudodifferential operator $V_{\ell r}$ is
defined via Fourier transform.

If $\bSi\in\mathfrak{M}^+(\Gamma)$ is given, then it becomes via
the pushforward $\cX_{\ell r\star}$ the localized distribution $\bSi_{\ell r}
:= \cX_{\ell r\star}\beta_{\ell r}\boldsymbol\Sigma$ with compact support
in $\mathcal{U}_{\ell r}\subset \mathbb{R}^{n-1}$ (see e.g.\ \cite[Lemma 5]
{HWZ1}), which can be Fourier transformed to $\widehat{\bSi}_{\ell r} (\bxi^\prime)$
on $\mathbb{R}^{n-1}$. The measures in $\mathfrak{M}^+(\Gamma)$
for which
\begin{equation}\label{hryu}
 \intli_{\mathbb{R}^{n-1}} |\bxi^\prime|^\beta
 |\widehat{\bSi}_{\ell r} (\bxi^\prime)|^2\d\bxi^\prime<\infty
\end{equation}
are precisely all those having finite Riesz energy (cf.~(\ref{eq:3.8A});
observe that the first summand on the left-hand side of~(\ref{eq:3.8A}) is the
dominant one). Let $\mathcal{E}_\alpha^\star(\Gamma)$ consist of all
$\bSi\in\mathfrak{M}^+(\Gamma)$ satisfying~(\ref{hryu}).

\begin{theorem}\label{th:3.4} There holds
\[
\mathcal{E}_\alpha^\star(\Gamma)=\mathcal{K}^{\beta/2}(\Gamma).
\]
\end{theorem}

\begin{proof}
Let $\boldsymbol\Sigma=\boldsymbol\sigma\in\mathcal{K}^{\beta/2}(\Gamma)$.
Then, by Theorem~\ref{th:x},
\[
(V_\beta\bsi,\bsi)_{L_2(\Gamma)}
=E_\alpha(\boldsymbol\sigma)<\infty.
\]
With \eqref{eq:3.8A} and Parseval's identity, together with
$\boldsymbol\Sigma_{\ell r}\in H^{\beta/2}(\mathcal{U}_{\ell r})$,
we obtain further
\begin{align*}
  &\intli_{\mathbb{R}^{n-1}} |\bxi^\prime|^\beta |\widehat{\bSi}_{\ell r} (\bxi^{\prime})|^2 \d\bxi^\prime\\
  	&\quad \le \big|(V_{\ell r}\boldsymbol\Sigma_{\ell r},\boldsymbol\Sigma_{\ell r})_{L^2(\mathcal{U}_{\ell r})}\big|
	+ \big|(A_{\ell r}^{\beta-2}\boldsymbol\Sigma_{\ell r},\boldsymbol\Sigma_{\ell r})_{L^2(\mathcal{U}_{\ell r})}\big|
	+ \big|(R_{\ell r}\boldsymbol\Sigma_{\ell r},\boldsymbol\Sigma_{\ell r})_{L^2(\mathcal{U}_{\ell r})}\big|\\
   &\quad \le c\|\boldsymbol\Sigma_{\ell r}\|_{H^{\beta/2}(\mathcal{U}_{\ell r})}^2
   + c'\|\boldsymbol\Sigma_{\ell r}\|_{H^{\beta/2+2-\beta}(\mathcal{U}_{\ell r})}
   \|\boldsymbol\Sigma_{\ell r}\|_{H^{\beta/2}(\mathcal{U}_{\ell r})}\\
   &\quad \le c''\|\boldsymbol\Sigma_{\ell r}\|_{H^{\beta/2}(\mathcal{U}_{\ell r})}^2 < \infty
\end{align*}
since $2-\beta/2 \ge \beta/2$. Consequently, it holds
$\boldsymbol\sigma=\bSi\in\mathcal{E}_\alpha^\star(\Gamma)$ and thus
$\mathcal{K}^{\beta/2}(\Gamma)\subseteq\mathcal E^\star_\alpha(\Gamma)$.

Next, suppose $\bSi\in\mathcal{E}_\alpha^\star(\Gamma)$; then by localization
inequality \eqref{hryu} holds. Since $  |\bxi^\prime|^\beta |(1+|\bxi^\prime|^2)^{-\beta/2}
\leq 1$ we find
\begin{align*}
    \|V_{\ell r}\bSi_{\ell r}\|_{H^{-\beta/2}(\mathcal{U}_{\ell r})}^2
        & =
    \intli_{\mathbb{R}^{n-1}} |\bxi^\prime|^{2\beta}(1+|\bxi^\prime|^2)^{-\beta/2}
    |\widehat{\bSi}_{\ell r}(\bxi^\prime)|^2\d \bxi^\prime \\
    & \leq
     \intli_{\mathbb{R}^{n-1}} |\bxi^\prime|^{\beta} |\widehat{\bSi}_{\ell r}(\bxi^\prime)|^2
     \d\bxi^\prime<\infty.
\end{align*}
With pullback to $\Gamma$ this implies that
\[
  \|V_\beta\bSi\|_{H^{-\beta/2}(\Gamma)}<\infty.
\]
Application of Theorem~\ref{th:3.3} then gives $\bSi\in
\mathcal{E}^+_\alpha(\Gamma)=\mathcal{K}^{\beta/2}(\Gamma)$,
which in view of the arbitrary choice of $\bSi\in\mathcal{E}_\alpha^\star(\Gamma)$
finally yields $\mathcal{K}^{\beta/2}(\Gamma)=\mathcal{E}_\alpha^\star(\Gamma)$.
\end{proof}

\begin{theorem}[see also {\cite[Theorem 4]{HWZ1}}]\label{t:2.3}
Let $\bSi\in\mathfrak{M}^+(\Gam)$ with $E_\alpha(\bSi) <\infty$.
Then there exists a sequence of absolutely continuous measures
$\bSi_k\in\mathcal{K}^{\beta/2}(\Gam)$, where $\d\bSi_k= \sum_{i\in I}
\alpha_i\varphi_k^i\,\d s$ with $\varphi_k^i\in C(\Gam_i)\cap H^{\beta/2}(\Gam_i)$
and $\varphi_k^i({\bf x})\geqslant 0$ for ${\bf x}\in\Gamma_i$, such that
$\{\bSi_k\}_{k\in\N}$ converges weakly and strongly in the Hilbert
space $H^{\beta/2}(\Gam)$ to $\bSi$, i.e.,
\[
\bSi_k(\bvfi)\overset{k\to\8}{\longrightarrow}\bSi(\bvfi)
\text{for all}\bvfi\in C(\Gam)\text{and}
\lim\limits_{k\to\8}\|\bSi-\bSi_k\|_{H^{\beta/2}(\Gam)}=0.
\]
\end{theorem}

\begin{proof}
For $\bSi$ we have $\bSi=\bsi\in\mathcal{K}^{\bet/2}(\Gam)$ due to
Theorem \ref{th:3.3}. Hence, since $C^\8(\Gam)\sbs H^{\bet/2}(\Gam)$ densely,
there exists a sequence $\bsi_k=\sum_{i\in I} \alpha_i\si_k^i\in C^\8(\Gam)$
with $\|\bsi-\bsi_k\|_{H^{\bet/2}(\Gam)}<\frac{1}{k}$ for all $k\in\mathbb{N}$.
We define $\bSi_k\in\fM^+(\Gam)$ by
\[
\d\bSi_k=\overline{\bsi}_k \d s\quad\text{with}\quad
\overline{\bsi}_k = \sum_{i\in I} \alpha_i\overline{\sigma}_k^i,
\]
where $\overline{\sigma}_k^i({\bf x}) :=\max\{0,\si_k^i({\bf x})\}$.
Then, $\overline{\sigma}_k^i\in C(\Gam_i)\cap H^{\beta/2}(\Gam_i)$
since $\overline{\si}_k^i$ is piecewise smooth and
\begin{align*}
\|\bsi-\overline{\bsi}_k\|_{H^{\beta/2}(\Gam)}
  &=\|\bsi-\bsi_k+\bsi_k-\overline{\bsi}_k\|_{H^{\beta/2}(\Gam)}\\
  &\le\|\bsi-\bsi_k\|_{H^{\beta/2}(\Gam)}
	+ \|\bsi_k-\overline{\bsi}_k\|_{H^{\beta/2}(\Gam)}\\
  &= \|\bsi-\bsi_k\|_{H^{\beta/2}(\Gam)}
	+ \sum_{i\in I}\|\underline{\sigma}_k^i\|_{H^{\beta/2}(\Gam_i)}.
\end{align*}
Herein, it holds $\underline{\sigma}_k^i({\bf x}) = \min\{0,\si_k^i({\bf x})\}
\le 0$ for all ${\bf x}\in\Gam_i$, particularly $\underline{\sigma}_k^i\in
C(\Gam_i)\cap H^{\beta/2}(\Gam_i)$. From $\sigma^i \ge 0$, it
immediately follows that
\[
  \|\underline{\sigma}_k^i\|_{H^{\beta/2}(\Gam_i)}
  \le\|\si^i-\si_k^i\|_{H^{\beta/2}(\Gam_i)}
\]
for all $i\in I$ and therefore
\[
  \|\bsi-\overline{\bsi}_k\|_{H^{\beta/2}(\Gam)}\le 2\|\bsi-\bsi_k\|_{H^{\beta/2}(\Gam)}
  \le \frac{2}{k}\overset{k\to\infty}\longrightarrow 0
\]
as desired.
\end{proof}

\section{The Gauss problem}\label{s:gauss}
\setcounter{equation}{0}
\setcounter{theorem}{0}
The Gauss variational problem is the problem of minimizing the Riesz
energy for particularly signed Borel measures on the given $(n-1)$-dimensional
manifold $\Gam\sbs\R^n$, in the presence of an external field.
Let
$\bg$ be a given continuous, positive function on $\Gam$ and
$\ba=(a_i)_{i\in I}$ a given vector with $a_i>0$, $i\in I$. Then, the
set of \emph{admissible charges\/} for the Gauss problem is defined as
\[
\cE_\al(\Gam,\ba,{\bf g}):= \Bigg\{\bmu\in\mathcal{K}^{\beta/2}(\Gam)\dopu \intli_{\Gamma_i}
g_i\d\mu^i=a_i\text{for all} i\in I\Bigg\}
\]
where we set $g_i := \bg_{|\Gam_i}$. Note that  the set
$\cE_\al(\Gam,\ba,{\bf g})$ is  convex and weakly and strongly
closed in $\mathcal{K}^{\beta/2}(\Gam)$.

The \emph{Gauss minimal energy problem\/} reads as follows (see \cite{Oh} and \cite{la}):
To given $\ba\in\R^{|I|}_+$, $\bof\in C(\Gam)$ and $\bg\in C(\Gam)$ such
that $\bg>{\bf 0}$, find the Borel measure $\bmu_0\in \cE_\al(\Gam,\ba,\bg)$
which is the minimizer of
\begin{equation}\label{eq:2kreuz}
\inf\limits_{\bmu\in\cE_\al(\Gam,\ba,\bg)}\bbG_{\bof}(\bmu)
=\bbG_{\bof}(\bmu_0)=:\bbG_{\bof}(\Gam,\ba,\bg)
\end{equation}
where the Gauss functional is given by
\[
\bbG_{\bf f}(\bmu):=E_\alpha(\bmu)-2\intli_\Gam \bof\d\bmu.
\]
Since $\bbG_{\bof}(\bmu)$ is on $\cE_\al(\Gamma,\ba,\bg)$ strictly
convex and weakly and strongly continuous, the Gauss problem
admits a unique solution $\bmu_0\in\cE_\al(\Gam,\ba,\bg)$.

Based on Theorem~\ref{th:x}, the minimization problem
\eqref{eq:2kreuz} can also be formulated as a variational
problem in $H^{\bet/2}(\Gam)$. Namely,  minimize the functional
\begin{equation}\label{eq:2.71x}
\V_{\bof}(\bvfi):=
(\bvfi,V_\bet\bvfi)_{L_2(\Gam)}-2(\bof,\bvfi)_{L_2(\Gam)},\quad
\bvfi\in H^{\bet/2}(\Gam),
\end{equation}
over the affine cone
\begin{align*}
\cK(\Gam,\ba,\bg)
&:=
\Bigg\{
\bvfi=\sumli_{i\in I}\al_i \vfi^i\text{where}
\vfi^i\in H^{\bet/2}(\Gam_i),\\
&\qquad
\vfi^i\ge 0\text{and} \intli_{\Gam_i}
g_i\vfi^i\d s=a_i>0\text{for all} i\in I\Bigg\}
\subset\mathcal{K}^{\beta/2}(\Gam)\subset H^{\bet/2}(\Gam)
\end{align*}
where $\mathbf{f}\in C(\Gamma)\komma\bg> 0$, $\bg\in C(\Gam)$
and ${\bf a}\in \mathbb{R}^{|I|}_+$ are given. This minimization
problem will be called the \emph{dual Gauss problem\/}.

\begin{theorem}\label{th:2.4}
To the unique solution $\bmu_0\in\cE_\al(\Gam,\ba,\bg)$ of the
Gauss problem \eqref{eq:2kreuz}, there corresponds a unique element
$\bvfi_0\in\cK(\Gam,\ba,\bg)\sbs H^{\bet/2}(\Gam)$ with the properties
\[
\bmu_0(\bvfi)=(\bvfi_0,\bvfi)_{L_2(\Gam)}\text{for all}\bvfi\in C^\8(\Gam)
\]
and
\[
\V_{\bof}(\bvfi_0)=\bbG_{\bof}(\bmu_0)=\bbG_{\bof}(\Gam,\ba,\bg).
\]
The element $\bvfi_0$ is the minimizer of the functional $\V_{\bof}$
over $\cK(\Gam,\ba,\bg)$, i.e.,
\begin{equation}\label{eq:2.*}
\V_{\bof}(\bvfi_0)=\min\limits_{\bvfi\in\cK(\Gam,\ba,\bg)}
\V_{\bof}(\bvfi)=:\V_{\bof}(\Gam,\ba,\bg).
\end{equation}
\end{theorem}

\begin{proof}
By Theorems \ref{th:x} and \ref{th:3.3}, to any Borel
measure $\bmu={\sum}_{i\in I}\al_i\mu^i\in\cE_\al(\Gam,\ba,\bg)$,
there corresponds a unique element $\bsi_\mu=\sum_{i\in I}\al_i
\si_\mu^i\in\mathcal K^{\bet/2}(\Gam)$ satisfying both \eqref{eq:stp}
and \eqref{eq:2.s}. Moreover, since $C^\infty(\Gamma)$ is dense in
$C(\Gamma)$, from \eqref{eq:2.s} we get
\[
(\si_\mu^i,g_i)_{L_2(\Gam_i)}=
\mu^i(g_i)= a_i\text{\ for all\ }i\in I.
\]
Hence, $\bsi_\mu\in\cK(\Gam,\ba,\bg)$.

Applying \eqref{eq:stp}, for these $\bmu$ and $\bsi_\mu$
we also obtain
\begin{equation}\label{eq:2.**}
\V_{\bof}(\bsi_\mu)= (V_\bet\bsi_\mu,\bsi_\mu)_{L^2(\Gam)}
-2(\bsi_\mu,\bof)_{L^2(\Gam)}=E_\al(\bmu)-2\bmu(\bof)
=\bbG_{\bof}(\bmu).
\end{equation}
Thus, the correspondence $\bmu\mapsto\bsi_\mu$
between $\cE_\al(\Gam,\ba,\bg)$ and $\cK(\Gam,\ba,\bg)$ is
one-to-one and satisfies~\eqref{eq:2.**}, which immediately implies
\[
\bbG_{\bof}(\Gam,\ba,\bg)=\V_{\bof}(\Gam,\ba,\bg).
\]

If now $\bmu_0$ is the (unique) solution of the  Gauss problem
\eqref{eq:2kreuz}, then $\bvfi_0$, the image of $\bmu_0$ under
this correspondence, is the unique solution of the minimizing
problem \eqref{eq:2.*}, and vice versa.
\end{proof}

\section{The particular case $\bal=\bnull$}\label{s:null}
\setcounter{equation}{0}
\setcounter{theorem}{0}
In the following, we will focus on the particular situation
$\alpha = 0$ from the potential theoretic point of view.
The double layer energy $E_0$ of a function $\bvfi\in
H^{\beta/2}(\Gamma)$, $\bet=1-\alpha=1$, which is
harmonic in $\Om$ (see \cite[Equation (1.2.17)]{H-W}) is
given by
\[
E_0(\bvfi):= (\bD\bvfi,\bvfi)_{L_2(\Gam)}
\]
with the hypersingular integral operator  $\bD$:
\begin{align*}
\bD\bvfi(\bx)
&:=
\pf \intli_{\Gam\sm\{\bx\}} k_{\bf D}(\bx,\by)\bvfi(\by)\d s_\by,\\
k_{\bf D}(\bx,\by)=k_{\bf D}(\by,\bx)
&\,=
c_n\left\{\f{\bn(\bx)\cd\bn(\by)}{|\bx-\by|^n}+n
\f{(\by-\bx)\cd\bn(\by)(\bx-\by)\cd\bn(\bx)}{|\bx-\by|^{n+2}}\right\}.
\end{align*}
The Hadamard partie finie integral operator is given by the finite
part with respect to $0<\de\to 0$ of
\begin{align*}
&\intli_{\by\in\Gam\wedge|\bx-\by|>\de}k_{\bf D}(\bx,\by)\bvfi(\by)\d s_\by\\
&\qquad=
\intli_{\by\in\Gam\wedge|\bx-\by|>\de}k_{\bf D}(\bx,\by)\big\{\bvfi(\by)-\bvfi(\bx)\big\}\d s_\by
+\intli_{\by\in\Gam\wedge|\bx-\by|>\de}k_{\bf D}(\bx,\by)\d s_\by\bvfi(\bx)
\end{align*}
if $\bvfi\in C^\8(\Gam)$.
The limit
\[
\lim_{\de\to 0}\!\!\!\intli_{0<\de<|\bx-\by|}
k_{\bf D}(\bx,\by)\big\{\bvfi(\by)-\bvfi(\bx)\big\}\d s_\by
=\pv\!\!\!\intli_{\by\in\Gam\sm\{\bx\}}k_{\bf D}(\bx,\by)\big\{\bvfi(\by)-\bvfi(\bx)\big\}\d s_\by
\]
exists (as a Cauchy principal value integral), whereas, from
$\int_{|\bx-\by|>\de}k_{\bf D}(\bx,\by)\d s_\by$, we have to
take the \emph{finite part\/}
\[
\lim_{\de\to 0}\pf \intli_{|\bx-\by|>\de}k_{\bf D}(\bx,\by)\d s_\by=:h(\bx).
\]
(For the evaluation of $h(\bx)$, see \eqref{eq:A.2}.) Hence,
we finally arrive at
\[
\bD\bvfi(\bx)=\pv\intli_{\Gam\sm\{\bx\}}k_{\bf D}(\bx,\by)\big\{\bvfi(\by)-\bvfi(\bx)\big\}\d s_\by
+h(\bx)\bvfi(\bx).
\]

\section{A perturbed minimal Riesz energy problem}
\label{s:mini}
Instead of considering the continuous minimal Riesz energy problem,
D.P.\ Hardin and E.B.\ Saff investigate in \cite{h-s2} the discrete Riesz
energy problem of minimizing $\overset{\circ}{E}_s(\boldsymbol{\omega}_N)$,
$s>d$, of the sum of a finite number $N$ of Dirac measures $\delta_{{\bf x}_i,N}$,
$\boldsymbol{\omega}_N:=\{{\bf x}_1,\ldots,{\bf x}_N\}$ being a set on a $d$-rec\-tifiable
manifold $A$. The energy $\overset{\circ}{E}_s(\boldsymbol{\omega}_N)$ is defined by
removing the self-interactions. For the sake of simplicity, assume that $A$ is compact
and has the positive Hausdorff measure $\mathcal{H}_d(A)>0$. Then, the infimum of
${\overset{\circ}{E}_s}(\boldsymbol{\omega}_N)$ over all point sets $\boldsymbol{\omega}_N
\subset A$ is attained at some $\boldsymbol{\omega}^\star_N:=\{{\bf x}^\star_1,
\ldots,{\bf x}^\star_N\}$. In particular, it is shown in \cite[Theorem 2.4]{h-s2} that
\[
\lim_{N\to\infty}\,\overset{\circ}{E}_s(\boldsymbol{\omega}_N^\star) N^{-(1+s/d)} =
C_{sd}/\mathcal{H}_d(A)^{s/d},
\]
where $C_{sd}$ is a constant independent of $A$ and defined
explicitly by the unit cube. Furthermore, in the weak-star topology of measures it holds
\[\frac{1}{N}\sum_{i=1}^N \delta_{{\bf x}_i,N}
  \to\frac{\mathcal{H}_d(\cdot)|_A}{\mathcal{H}_d(A)}
  \quad\text{as $N\to\infty$}.
\]
If $A$ is a bi-Lipschitz image of a single compact set in $\mathbb{R}^d$,
then the separation estimate \cite[Eq.~(16)]{h-s2} has also been established for an optimal
$N$ point $s$-energy configuration $\boldsymbol{\omega}^\star_N$ for $A$.

Now, for our $(n-1)$-dimensional manifold $\Gam=\bigcup_{i\in I}\Gamma_i$,
every compact smooth $\Gamma_i$ immersed into $\mathbb{R}^n$,
satisfies all the assumptions on $A$ with $s=n-\alpha$, $-1<\alpha<1$,
$d'=n>d=n-1$. Following the inspiration of these results, in our continuous
setting, this corresponds to integrating for a small $\delta>0$ over
$(\Gam\x\Gam)\sm \{|\bx-\by|\le\de\}$, i.e., by
cutting out a set with $|\bx-\by|\le\de$ near the singularity. In order to explain
the computations in the next section, we shall focus first on the following
perturbation problem which, for $0<\vep = \vep(\delta)\to 0$, is essentially
the minimization problem we will finally get.

\begin{theorem}\label{th:4.1}
For $\vep>0$ sufficiently small, consider
the minimization problem:
\begin{equation}\label{eq:4.1}
\big(V_\bet\bsi,\bsi\big)_{L_2(\Gam)}
+\frac{1}{\vep}(\bsi,\bsi\big)_{L_2(\Gam)}
-2(\bof,\bsi)_{L_2(\Gam)}
\to \min
\end{equation}
subject to
\begin{equation}\label{eq:4.2}
\intli_{\Gam_i} g_i\si^i\d s= a^i,\ j\in I.
\end{equation}
Let the given data satisfy the additional conditions:
\begin{equation}\label{eq:4kreuz}
\bof\in H^{\bet/2}(\Gam),\
\bg\in H^{\f{3}{2}\bet}(\Gam)
\text{and}0<a^i\in\R,\ i\in I.
\end{equation}
Then, the minimizer $\bsi_\vep^\star\in L_2(\Gam)$ admits the asymptotic
expansion
\begin{equation}\label{eq:4.3}
\bsi^\star_\vep
=
\bsi_0+\vep\bsi_1+\vep^2\bsi_2
\text{satisfying}\para\bsi_j\para_{\cH_\vep}\le c,\ j=0,1,2
\end{equation}
with a constant $c>0$ independent of $\vep$, and where
\[
\bcH_\vep:=\big\{\bvfi\in H^{\bet/2}(\Gam)\text{with}
\para\bvfi\para_{\cH_\vep}^2:=
(\vep V_\bet\bvfi,\bvfi)_{L_2(\Gam)}+\|\bvfi\|^2_{L_2(\Gam)}\big\}
\sbs H^{\bet/2}(\Gam).
\]
In particular, with $f_k:=\bof_{|\Gam_k}$, it holds
\begin{align*}
\bsi_0
&=
(\al_k\si_0^k)_{k\in I} =
\Big(\al_k g_k a^k (g_k,g_k)^{-1}_{L_2(\Gam_k)}\Big)_{k\in I},\\
\bsi_1
&=
(\al_k\si_1^k)_{k\in I}=
\Big(\al_k g_k(g_k,g_k)^{-1}_{L_2(\Gam_k)}(g_k,V_\bet\bsi_0-f_k)_{L_2(\Gam_k)}
\Big)_{k\in I}-V_\bet\bsi_0+{\bf f}.
\end{align*}
\end{theorem}

\begin{proof}
The quadratic form in \eqref{eq:4.1} induces for $\vep>0$ the
$\vep$-dependent family of Hilbert spaces $\bcH_\vep$. Let
us denote by $\bcH_\vep\p$ the dual space to $\bcH_\vep$
whose norm is then defined by
\[
\para\bof\para_{\cH\p_\vep}:= \sup_{\bw\in H^{\bet/2}(\Gam)\setminus\{0\}}
\f{(\bof,\bw)_{L_2(\Gam)}}{\para\bw\para_{\cH_\vep}},
\]
satisfying the estimate
\begin{equation}\label{eq:3herz}
\para\bof\para_{\cH\p_\vep}\le\sup_{\bw\in H^{\bet/2}(\Gam)\setminus\{0\}}
\f{(\bof,\bw)_{L_2(\Gam)}}{\|\bw\|_{L_2(\Gam)}}=\|\bof\|_{L_2(\Gam)}
\end{equation}
since $H^{\bet/2}(\Gam)\hookrightarrow L_2(\Gam)$ densely and the unit
ball in $\bcH_\vep$ is contained in the ball $\|\bw\|_{L_2(\Gam)}\le 1$.

The problem \eqref{eq:4.1} can also be written as to minimize
\[
\bJ_\vep(\bsi):=\ha\para\bsi\para_{\bcH_\vep}^2-\vep(\bof,\bsi)_{L_2(\Gam)}
\]
subject to \eqref{eq:4.2}. It possesses
the Lagrangian
\[
\bL_\lam (\bsi):=
\ha\para\bsi\para^2_{\cH_\vep}-\vep(\bof,\bsi)_{L_2(\Gam)}
+\sumli_{j\in I} \al_j \lam_j\big(a^j-(g_j,\si_j)_{L_2(\Gam_j)}\big),
\]
where $\bsi=\sumli_{j\in I}\al_j\si^j$ and $\al_j\si^j\d s=\bSi_{|\Gam_j}$.
Thus, the necessary conditions at the minimum read as
\[
\pa_{\si^k} \bL_\lam(\bsi_\vep)= \vep\al_k V_\bet\si_\vep^k +
\al_k \si_\vep^k - \vep \al_k f_k-\al_k\lam_k g_k=0,
\]
or
\begin{equation}\label{eq:4.4}
\vep V_\bet\si_\vep^k +\si_\vep^k-\vep f_k=\lam_k g_k\text{and}
(g_k,\si_\vep^k)_{L_2(\Gam_k)}=a^k,\ k\in I.
\end{equation}
Here, $\blam\neq\bnull$ since the constraints \eqref{eq:4kreuz} are
always active as it follows from \eqref{eq:3herz} and will also be
seen below.

For $\bsi$ and $\blam$, we insert the expansion
\eqref{eq:4.3} into \eqref{eq:4.4} and obtain the system
\begin{gather*}
\vep V_\bet\si_0^k+\si_0^k-\vep f_k
+
\vep^2 V_\bet \si_1^k+\vep\si_1^k+\vep^3 V_\bet
\si_2^k+\vep^2\si_2^k
=
(\lam_0^k+\vep \lam_1^k+\vep^2\lam_2^k) g_k\text{on}\Gam_k,\\
(g_k,\si_0^k)_{L_2(\Gam_k)}
+
\vep(g_k,\si_1^k)_{L_2(\Gam_k)}+\vep^2 (g_k,\si_2^k)_{L_2(\Gam_k)}
=
a^k,\ k\in I.
\end{gather*}
Equating equal order terms in $\vep$ yields with $\lam_0^j>0$:

\bigskip
\noindent
\underline{Order $\bvep^0$:} We find
\[
\si_0^j=g_j\lam_0^j\text{and}(g_j,\si_0^j)_{L_2(\Gam_j)}=
(g_j,g_j)_{L_2(\Gam_j)}\lam_0^j =a^j>0.
\]
Thus, it follows
\begin{equation}\label{eq:4.5}
\lam_0^j=(g_j,g_j)_{L_2(\Gam_j)}^{-1}a^j\text{and}
\si_0^k=g_k a^k (g_k,g_k)_{L_2(\Gam_k)}^{-1}
\in H^{\f{3}{2}\bet}(\Gam_k).
\end{equation}
The assumptions \eqref{eq:4kreuz} imply the  properties
$\si_0^k>0$. Moreover,
\[
\para\si_0^k\para_{\cH_\vep}\le
c_1\|\si_0^k\|_{H^{\f{3}{2}\bet}(\Gam_k)}
\le c.
\]

\bigskip
\noindent
\underline{Order $\bvep^1$:} It holds
\[
-f_k+\si_1^k
=
\lam_1^k g_k-V_\bet \si_0^k
\]
and
\[
(V_\bet\si_0^k,g_k)_{L_2(\Gam_k)}
-(f_k,g_k)_{L_2(\Gam_k)}
+ (\si_1^k,g_k)_{L_2(\Gam_k)}
=
\lam_1^k(g_k,g_k)_{L_2(\Gam_k)}.
\]
This yields with $(\si_1^k,g_k)_{L_2(\Gam_k)}=0$:
\begin{equation}\label{eq:4.8a}
\begin{array}{c}
\lam_1^k
=
(g_k,g_k)^{-1}_{L_2(\Gam_k)}(V_\bet\si_0^k-f_k,g_k)_{L_2(\Gam_k)},\\[2ex]
\si_1^k
=
g_k(g_k,g_k)^{-1}_{L_2(\Gam_k)}(V_\bet\si_0^k-f_k,g_k)_{L_2(\Gam_k)}+f_k
-V_\bet \si_0^k\in H^{\bet/2}(\Gam_k).\end{array}
\end{equation}
Hence,
\[
\para\si_1^k\para_{\cH_\vep}\le c_2\|\si_0^k\|_{H^{\bet/2}(\Gam_k)}\le
c.
\]

\bigskip
\noindent
\underline{Order $\bvep^2$:} We derive the identities
\begin{equation}\label{eq:4.8b}
\begin{array}{c}
V_\bet \si_1^k+(\vep V_\bet\si_2^k+\si_2^k)
=
\lam_2^k g_k,\\[2ex]
\vep(V_\bet\si_2^k,g_k)_{L_2(\Gam_k)}+(V_\bet
\si_1^k,g_k)_{L_2(\Gam_k)}
=
\lam_2^k (g_k,g_k)_{L_2(\Gam_k)}
\end{array}
\end{equation}
since $(\si_2^k,g_k)_{L_2(\Gam_k)}=0$.
Therefore, we conclude
\begin{equation}\label{eq:4.8c}
\begin{aligned}
&(\vep V_\bet+I)_{|\Gam_k}\si_2^k=
g_k(g_k,g_k)^{-1}_{L_2(\Gam_k)}\\
&\qquad\cdot\big\{\vep(V_\bet\si_2^k,g_k)_{L_2(\Gam_k)}
+(V_\bet\si_1^k,g_k)_{L_2(\Gam_k)}\big\}-V_\bet\si_1^k\in H^{-\bet/2}(\Gam_k).
\end{aligned}
\end{equation}

\bigskip
For every fixed $\vep>0$ sufficiently small, the mapping
$
\vep V_\bet+I\dopu \bcH_\vep\to\bcH_\vep\p
$
defines an isomorphism due to \eqref{eq:2.2}
and the Lax--Milgram lemma. Therefore, \eqref{eq:4.8c}
amounts to the estimate
\begin{align*}
\para\bsi_2\para_{\bcH_\vep}
&\le
c\,\bigpara
\big(g_k(g_k,g_k)^{-1}_{L_2(\Gam_k)}\big\{\vep(V_\bet\si_2^k,g_k)_{L_2(\Gam_k)}
+(V_\bet\si_1^k,g_k)_{L_2(\Gam_k)}\big\}\big)_{k\in I}
-V_\bet\bsi_1\bigpara_{\bcH_\vep\p}\\
&\le
c\p\vep\bigg(\sumli_{k\in I}(V_\bet\si_2^k,g_k)_{L_2(\Gam_k)}\bigg)
+c\p\|V_\bet\bsi_1\|_{H^{-\bet/2}(\Gam)}
\end{align*}
with a constant $c\p$ depending on $\Gam$ and $\bg$
due to \eqref{eq:3herz} but \emph{not\/} on $\vep$.
With $\bg\in H^\bet(\Gam)$ and $V_\bet$ being a
pseudodifferential operator of order $\bet$ on $\Gam$,
we further have
\[
\bigg|\sumli_{k\in I}(V_\bet \si_2^k,g_k)_{L_2(\Gam_k)}\bigg|
\le c\pp\|\bsi_2\|_{L_2(\Gam)}\|\bg\|_{H^\bet(\Gam)},
\]
implying that
\begin{align*}
\para\bsi_2\para_{\bcH_\vep}
&\le
 c^{\prime\prime\prime}\vep \|\bsi_2\|_{L_2(\Gam)}
+c_\bv^{\prime}\|\bsi_1\|_{H^{\bet/2}(\Gam)}\\
&\le
 c^{\prime\prime\prime}\vep \|\bsi_2\|_{\bcH_\vep}
+c_\bv^{\prime}\|\bsi_1\|_{H^{\bet/2}(\Gam)}.
\end{align*}
Consequently, since the constants do not depend on $\vep$,
there exists an $\vep_0>0$ such that
\[
\para\bsi_2\para_{\cH_\vep}\le
c_\bv\|\bsi_1\|_{H^{\bet/2}(\Gam)} \text{for all} 0<\vep<\vep_0
\]
and $c_\bv$ independent of $\vep$.

With $\bof\in H^{\bet/2}(\Gam)$, $\bg\in H^{\f{3}{2}\bet}(\Gam)$,
we find $\bsi_0\in H^{\f{3}{2}\bet}(\Gam)$ and $\bsi_1\in H^{\bet/2}
(\Gam)$. Hence, \eqref{eq:4.3} is justified which completes the proof
of Theorem \ref{th:4.1}.
\end{proof}

With the help of the previous theorem, we immediately
find the following asymptotic behaviour of the minimizer
$\bsi_\vep^\star$ if $\vep$ tends to zero.

\begin{cor}\label{th:4.2}
Under the assumptions in Theorem \ref{th:4.1},
we find that
\[
\|\bsi_\vep^\star-\bsi_0\|_{L_2(\Gam)}\le c\vep
\overset{\vep\to0}{\longrightarrow} 0
\]
with some constant $c$, independent of $\bsi_0$ and
$\vep>0$, where $\bsi_\vep^\star$ is the minimizer
\eqref{eq:4.3} of \eqref{eq:4.1} for $\vep>0$.
\end{cor}

\begin{proof}
Since $\bsi_\vep^\star=\bsi_0+\vep\bsi_1+\vep^2\bsi_2$,
with \eqref{eq:4.3} we find
\[
\|\bsi_\vep^\star-\bsi_0\|_{L_2(\Gam)}\le
\|\bsi_\vep^\star-\bsi_0\|_{\mathcal{H}_\vep}\le
c\p(\vep+\vep^2)\le 2c\p\vep
\]
as proposed.
\end{proof}

\section{Riesz minimal energy without finite part reduction}
\label{s:comp}
We consider next the punched hypersingular Riesz potential
which is defined by integrating for a small $\delta > 0$ only over
$(\Gam\x\Gam)\sm \{|\bx-\by|\le\de\}$, i.e., by cutting out a set
with $|\bx-\by|\le\de$ near the singularity. Thus, the corresponding
Riesz energy is defined as
\[
\J_\de(\bmu)=
\iint\limits_{\Gam\x\Gam\wedge 0<\de\le|\bx-\by|} |\bx-\by|^{\al-n}
\d\bmu(\bx)\otimes\d\bmu(\by)-2\intli_\Gam \bof(\bx)\d\bmu(\bx).
\]
In view of Theorems \ref{th:3.3}, \ref{t:2.3} and \ref{th:2.4}, the
associated minimal Riesz energy problem is then equivalent to
minimizing the punched functional
\[
\J_\de(\bvfi)=
\iint\limits_{\Gam\x\Gam\wedge 0<\de\le|\bx-\by|}|\bx-\by|^{-m-\bet}
\bvfi(\by)\bvfi(\bx) \d s_\by\d s_\bx -2(\bof,\bvfi)_{L_2(\Gam)},
\]
where $m=n-1$ and $\bet=1-\al\in(0,2)$, over the affine cone
$\cK(\Gamma,{\bf a},{\bf g})$. Then, the measures satisfy
$\d\bmu(\bx)=\bvfi(\bx)\d s$ with $\d s$ being the surface
measure on $\Gam$.

For $\J_\de(\vfi)$ one has the following monotonicity property.

\begin{lem}\label{l:4.2}
Let $0<\de_1<\de_2$ and $\bvfi_{\de_1}^\star,\bvfi_{\de_2}^\star\in
\cK(\Gamma,{\bf a},{\bf g})$ be the minimizers of $\J_{\de_1}$ and $\J_{\de_2}$,
respectively. Then, it holds that
\begin{equation}\label{eq:4.10}
\J_{\de_1}(\bvfi_{\de_1}^\star)\ge \J_{\de_2}(\bvfi_{\de_1}^\star)
\ge\J_{\de_2}(\bvfi_{\de_2}^\star).
\end{equation}
\end{lem}

\begin{proof}
Since $\de_1<\de_2$, the minimizer $\bvfi_{\de_1}^\star=\sum_{j\in
I}\al_j\vfi_{\de_1}^{\star j}$ with $\vfi_{\de_1}^{\star j}\ge 0$ is an
admissible element for minimizing $\J_{\de_2}$.
In particular, it holds
\begin{align*}
\J_{\de_1}(\bvfi_{\de_1}^\star)
&=
\iint\limits_{|\bx-\by|>\de_1} |\bx-\by|^{-m-\bet}\bvfi_{\de_1}^\star(\bx)
\bvfi_{\de_1}^\star(\by)\d s_\by\d s_\bx
-2(\bof,\bvfi_{\de_1}^\star)_{L_2(\Gam)}\\
&\ge
\iint\limits_{|\bx-\by|>\de_2} |\bx-\by|^{-m-\bet}\bvfi_{\de_1}^\star(\bx)
\bvfi_{\de_1}^\star(\by)\d s_\by\d s_\bx
-2(\bof,\bvfi_{\de_1}^\star)_{L_2(\Gam)}
=
\J_{\de_2}(\bvfi_{\de_1}^\star).
\end{align*}
We further find
\[
\J_{\de_2}(\bvfi_{\de_1}^\star)\ge \inf \J_{\de_2}(\bvfi)
=\J_{\de_2}(\bvfi_{\de_2}^\star),
\]
as proposed in \eqref{eq:4.10}.
\end{proof}

In order to see the relation between $\J_\de(\bvfi)$ and $\V_{\bof}(\bvfi)$
in \eqref{eq:2.71x}, let us introduce the compensating quadratic functional
\begin{equation}\label{eq:4.11}
\bP_\de(\bvfi)
:=\intli_\Gam\Bigg\{\pf
\intli_{|\bx-\by|\le\de}|\bx-\by|^{-m-\bet}\bvfi(\by)\d s_\by\Bigg\}
\bvfi(\bx)\d s_\bx.
\end{equation}
Then, we obtain
\begin{align*}
\bP_\de(\bvfi)+\J_\de(\bvfi)
&=
\intli_\Gam\Bigg\{\pf
\intli_\Gam |\bx-\by|^{-m-\bet}\bvfi(\by)\d s_\by\Bigg\}
\bvfi(\bx)\d s_\bx-2(\bof,\bvfi)_{L_2(\Gam)}\\
&=
(V_\bet\bvfi,\bvfi)_{L_2(\Gam)}-2(\bof,\bvfi)_{L_2(\Gam)}
=\V_{\bof}(\bvfi),
\end{align*}
and thus
\[
\J_\de(\bvfi)
=
\V_{\bof}(\bvfi)-\bP_\de(\bvfi)=(V_\bet\bvfi,\bvfi)_{L_2(\Gam)}
-2(\bof,\bvfi)_{L_2(\Gam)}-\bP_\de(\bvfi).
\]
For the corresponding functional $\bP_\de$, there holds

\begin{lem}\label{l:4.3}
Let $\bvfi\in\cK(\Gamma,{\bf a},{\bf g})$. Then
\[
\bP_\de(\bvfi)
=
-\f{1}{\bet}\f{1}{c_m}
\de^{-\bet}\|\bvfi\|^2_{L_2(\Gam)}+\bP_\de\p(\bvfi),
\]
where $ \bP_\de\p(\bvfi)$ satisfies
\[
|\bP_\de\p(\bvfi)|
\le
c\|\bvfi\|_{H^{\bet/2}(\Gam)}^2
\]
with a constant $c$ independent of $\de$.
Moreover
\begin{equation}\label{eq:4.14}
\lim\limits_{\de\to 0}\bP_\de\p(\bvfi)=0 \text{for every}
\bvfi\in H^{\bet/2}(\Gam).
\end{equation}
\end{lem}

\begin{proof}
Using Martensen's coordinates of $\Gam$ in the vicinity of $\bx\in\Gam$
(see Theorem~\ref{t:3.3} in Appendix \ref{a:b}), for every $\bvfi\in
C^\8(\Gam)$,
we have
\begin{align*}
\bP_\de(\bvfi)
&=
\intli_\Gam\Bigg\{\pf \intli_{0<r\le\de\wedge |\bTheta|=1}
\bvfi(\bx+r\bTheta)r^{-\bet-1}\d r\wedge\d\bom\\
&\qquad\quad
+ \intli_{0<r\le\de\wedge |\bTheta|=1} (\bx+r\bTheta) a(\bx,r) r^{-\bet+1}
\d r\wedge\d\bom\Bigg\} \bvfi(\bx)\d s_\bx\\
&=
\Bigg\{\pf\intli_{0<r\le\de\wedge |\bTheta|=1} r^{-\bet-1}
\d r\wedge\d\bom\bigg\}\intli_\Gam |\bvfi(\bx)|^2\d s_\bx\\
&\quad
+ \intli_\Gam \Bigg\{\pf \intli_{0<r\le\de\wedge |\bTheta|=1}
\big\{\bvfi(\bx+r\bTheta)-\bvfi(\bx)\big\}r^{-\bet-1}\d r\wedge\d\bom\Bigg\}
\bvfi(\bx)\d s_\bx\\
&\quad
+ \intli_\Gam \Bigg\{\pf \intli_{0<r\le\de\wedge |\bTheta|=1}
\bvfi(\bx+r\bTheta) a(\bx,r) r^{-\bet+1}\d r\wedge\d\bom\Bigg\}
\bvfi(\bx)\d s_\bx.
\end{align*}
Since
\[
\pf\intli_{0<r\le\de\wedge |\bTheta|=1}r^{-\bet-1}\d r\wedge
\d\bom=-\f{1}{\bet}\f{1}{c_m}\de^{-\bet}
\]
and $\bP_\de(\bvfi)$ in
\eqref{eq:4.11} is symmetric, we find
\[
\bP_\de(\bvfi)=-(\bet c_m)^{-1}\de^{-\bet}\|\bvfi\|^2_{L_2(\Gam)}
+\bP_\de\p (\bvfi).
\]
Herein,  $\bP\p_\de(\vfi)$ is given by
\begin{align}\label{eq:4.15}
2\bP\p_\de(\bvfi)
&=
\iint\limits_{|\bx-\by|\le\de}\big\{\bvfi(\bx)\big(\bvfi(\by)-\bvfi(\bx)\big)+
\bvfi(\by)\big(\bvfi(\bx)-\bvfi(\by)\big)\big\}|\bx-\by|^{-\bet-m}
\d s_\by \d s_\bx\no\\
&\qquad-
\iint\limits_{|\bx-\by|\le\de}\big(\bvfi(\by)-\bvfi(\bx)\big) a(\by,r)
r^{-\bet +1}\d r\wedge\d\bom\bvfi(\bx)\d s_\bx\no\\
&\qquad-
\iint_{|\bx-\by|\le\de}\big(\bvfi(\bx)-\bvfi(\by)\big)a(\bx,r)
r^{-\bet+1}\d r\wedge\d\bom\,\bvfi(\by) \d s_\by\no\\
&\qquad+
\iint\limits_{|\bx-\by|\le\de}\bvfi(\by) a(\bx,r)
r^{-\bet+1}\d r\wedge\d\bom\,\bvfi(\bx)\d s_\bx\no\\
&\qquad+
\iint\limits_{|\bx-\by|\le\de}\bvfi(\bx) a(\by,r)
r^{-\bet+1}\d r\wedge\d\bom\bvfi(\by)\d s_\by.
\end{align}
We rewrite $\bP\p_\de(\bvfi)$ according to
\begin{align*}
\bP\p_\de(\bvfi)
&=
\ha\iint\limits_{|\bx-\by|\le\de}|\bvfi(\bx)-\bvfi(\by)|^2
|\bx-\by|^{-\bet-m}\d s_\by\d s_\bx\\
&\qquad+
\iint\limits_{|\bx-\by|\le\de} b(\bx,\by)\bvfi(\bx)\bvfi(\by)\d s_\by\d s_\bx
+\iint\limits_{|\bx-\by|\le\de} c(\bx,\by)\bvfi(\bx)^2\d s_\by\d s_\bx,
\end{align*}
where $b(\bx,\by)$ and $c(\bx,\by)$ are kernels which posess
pseudohomogeneous expansions of degree $-\bet-m+1$.
This means that
\begin{align*}
\bP_\de\p(\bvfi)
&=
\ha \iint\limits_{|\bx-\by|\le\de}|\bvfi(\bx)-\bvfi(\by)|^2 |\bx-\by|^{-\bet-m}
\d s_\by\d s_\bx\no\\
&\qquad+
\intli_\Gam (\bB_{\bet-1}\bvfi)(\bx)\bvfi(\bx)\d s_\bx+
\intli_\Gam (\bC_{\bet-1}1)\bvfi(\bx)^2\d s_\bx
\end{align*}
with classical pseudodifferential operators $\bB_{\bet-1}$
and $\bC_{\bet-1}$ of degree $\bet-1$ on $\Gam$. Since
the constant charge $1$ is a smooth function on
$\Gam$ and $\bB_{\bet-1}\dopu H^{\bet/2}(\Gam)\to
H^{-\bet/2+1}(\Gam)\hookrightarrow H^{-\bet/2}(\Gam)$
for $\bvfi\in H^{\bet/2}(\Gam)$, we finally arrive at
\[
\big|\bP\p_\de(\bvfi)\big|
\le
c_1\|\bvfi\|^2_{H^{\bet/2}(\Gam)} +c_2\|\bvfi\|^2_{L_2(\Gam)}
\le
c\|\bvfi\|^2_{H^{\bet/2}(\Gam)},
\]
as proposed.\footnote{For $\bet=1$, $\bB_0$ is a singular
Mikhlin--Calderon integral operator with principal part
$b(\bx,\bx)\bTheta(\om)$ which satisfies the Mikhlin condition
$b(\bx,\bx)\int_{|\bTheta|=1}\bTheta(\om)\d\bom=0$.}

In order to show \eqref{eq:4.14}, consider first $\bvfi\in C^\8(\Gam)$
and use Taylor's expansion about $\bx\in\Gam$ in \eqref{eq:4.15}.
Then, all the integrals on the right hand side are weakly singular
tending to zero with $\de\to 0$. For $\bvfi\in H^{\bet/2}(\Gam)$
approximate $\bvfi$ by $\bvfi_\vep\in C^\8(\Gam)$, satisfying
$\|\bvfi-\bvfi_\vep\|_{H^{\bet/2}(\Gam)}<\vep$. Then
\begin{align*}
&\big|\bP_\de\p(\bvfi)-\bP_\de\p(\bvfi_\vep)\big|\\
&\quad\le\ha\iint\limits_{|\bx-\by|\le\de}\big\{|\bvfi(\bx)-\bvfi(\by)|^2-
|\bvfi_\vep(\bx)-\bvfi_\vep(\by)|^2\big\}|\bx-\by|^{-\bet-m}\d s_\by
\d s_\bx\\
&\qquad+\big|(\bB_{\bet-1}\bvfi,\bvfi)_{L_2(\Gam)}-
(\bB_{\bet-1}\bvfi_\vep,\bvfi_\vep)_{L_2(\Gam)}\big|\\
&\qquad+\|\bC_{\bet-1}1\|_{L_2(\Gam)}\big|\|\bvfi\|_{L_2(\Gam)}^2-
\|\bvfi_\vep\|_{L_2(\Gam)}^2\big|.
\end{align*}
With
\[
\|\bvfi\|_{H^{\bet/2}(\Gam)}^2-\|\bvfi_\vep\|_{H^{\bet/2}(\Gam)}^2
\le 3\|\bvfi\|_{H^{\bet/2}(\Gam)}\,\|\bvfi-\bvfi_\vep\|_{H^{\bet/2}(\Gam)}
\]
for $\|\bvfi_\vep\|_{H^{\bet/2}(\Gam)}\le 2\|\bvfi\|_{H^{\bet/2}(\Gam)}^2$,
one has
\begin{align*}
&\big|\bP_\de\p(\bvfi)-\bP_\de\p(\bvfi_\vep)\big|\\
&\ \le\f{3}{2}
\Bigg\{\ \iint\limits_{|\bx-\by|\le\de}\big|\bvfi(\bx)-\bvfi_\vep(\bx)\big|
\big|\bvfi(\by)-\bvfi_\vep(\by)\big||\bx-\by|^{-\bet-m}\d s_\by\d s_\bx
\Bigg\}^\ha\|\bvfi\|_{H^{\bet/2}(\Gam)}\\
&\qquad +
\big|\big(\bB_{\bet-1}(\bvfi-\bvfi_\vep),\bvfi\big)_{L_2(\Gam)}\big|
+\big|\big(\bB_{\bet-1}\bvfi_\vep,(\bvfi-\bvfi_\vep)\big)_{L_2(\Gam)}\big|\\
&\qquad +
3\|\bC_{\bet-1}1\|_{L_2(\Gam)}\|\bvfi\|_{H^{\bet/2}(\Gam)}
\|\bvfi-\bvfi_\vep\|_{H^{\bet/2}(\Gam)}\\
&\ \le
c\|\bvfi\|_{H^{\bet/2}(\Gam)} \|\bvfi-\bvfi_\vep\|_{H^{\bet/2}(\Gam)}
\le c\vep \|\bvfi\|_{H^{\bet/2}(\Gam)}.
\end{align*}
Then
\[
\ov{\lim_{\de\to 0}}\big|\bP_\de\p(\bvfi)\big|\le
c\vep\|\bvfi\|_{H^{\bet/2}(\Gam)}
\]
for any $\vep> 0$ which implies \eqref{eq:4.14}.
\end{proof}

Since
\begin{equation}\label{eq:connection}
\begin{aligned}
\J_\de(\bvfi)
&=
\iint\limits_{0<\de\le|\bx-\by|}|\bx-\by|^{-\bet-m}\bvfi(\by)\bvfi(\bx)
\d s_\by\d s_\bx-2(\bof,\bvfi)_{L^2(\Gam)}\\
&=
\f{1}{\vep}\|\bvfi\|^2_{L_2(\Gam)}+ (V_\bet\bvfi,\bvfi)_{L_2(\Gam)}
-2(\bof,\bvfi)_{L_2(\Gam)}-\bP\p_\de(\bvfi),
\end{aligned}
\end{equation}
where $\vep=\bet c_m\de^\bet\to 0$ for $\de\to 0$,
Lemma \ref{l:4.2} implies

\begin{cor}\label{c:4.4}
Let $\bet_0<2$ and $\de_0>0$ be given. Then, there exist
positive constants $c,c\p>0$ such that the estimate
\[
\J_\de(\bvfi)-\f{1}{\bet}\,\f{1}{c_m}
  \de^{-\bet}\|\bvfi\|^2_{L_2(\Gam)}\le c\|\bvfi\|^2_{H^{\bet/2}(\Gam)}
+2\|\bof\|_{L_2(\Gam)}\|\bvfi\|_{L_2(\Gam)}\le c\p
\]
holds uniformly for $\bvfi\in\cK^{\bet/2}(\Gam)$, $0<\de\le\de_0$ and
$0\le\bet\le\bet_0$.
\end{cor}

The functional \eqref{eq:connection} coincides with the
functional from \eqref{eq:4.1} except for the perturbation
term $\bP\p_\de(\bvfi)$. Hence, we can proceed in complete
analogy to the proof of Theorem \ref{th:4.1}.

The Lagrangian to the punched energy functional reads as
\begin{align*}
\overset{\circ}{\bL}_\lam(\bsi)&:=
\ha \big((\vep V_\bet\bsi+\bsi),\bsi\big)_{L_2(\Gam)}
-\vep\bP\p_\de (\bsi)-\vep(\bof,\bsi)_{L_2(\Gam)}\\
&\hspace*{3.5cm}+\sum_{j\in I} \al_j\lam_j \big(a^j-(g_j,\si^j)_{L_2(\Gam_j)}\big),
\end{align*}
where the first order necessary optimality condition
is given by
\begin{gather*}
\pa_{\si^k}\overset{\circ}{\bL}_\lam(\bsi)=
\vep\al_k V_\bet \si_\vep^k+ \al_k\si_\vep^k-\vep \al_k\bP_\de\pp
\si_\vep^k-\vep\al_k f_k-\al_k\lam_k g_k=0,\\
(g_k,\si_\vep^k)_{L_2(\Gam_k)}=a^k,\ k\in I.
\end{gather*}
Again, $\blam\neq \bnull$ since the constraints \eqref{eq:4kreuz}
are active. Here,
\begin{align*}
\bP_\de\pp\si_\vep^k =\pa_{\si_\vep^k}\bP_\de\p (\bsi_\vep)
&= \intli_{|\bx-\by|\le\de}\big(\si_\vep^k(\by)-\si_\vep^k(\bx)\big)
|\bx-\by|^{-m-\bet}\d s_\by\\
&\qquad\qquad+\bB_{\bet-1}\si_\vep^k
+\bB_{\bet-1}^\star\si_\vep^k+2c_{\bet-1}(\bx)\si_\vep^k(\bx)
\end{align*}
with
\[
c_{\bet-1}(\bx)
=\intli_{\by\in\Gam\wedge|\bx-\by|\le\de}\!\!\!\!c(\bx,\by)
\d s_\by
\]
and $\bB_{\bet-1}$ and $\bB_{\bet-1}^\star$ being
bounded linear operators:
\begin{align*}
\bB_{\bet-1}\si_\vep^k
&=\intli_{\by\in\Gam\wedge|\cdot-\by|\le\de}b(\cdot,\by)
\si_\vep^k(\by)\d s_\by:H^{\f{3}{2}\bet}(\Gam)\to H^{\ha\bet}(\Gam),\\
\bB_{\bet-1}^\star\si_\vep^k
&= \intli_{\by\in\Gam\wedge|\cdot-\by|\le\de}b(\by,\cdot)
\si_\vep^k(\by)\d s_\by:H^{\ha\bet}(\Gam)\to H^{-\ha\bet}(\Gam).
\end{align*}

Under the assumptions of Theorem \ref{th:4.1} and as
in the proof of Theorem \ref{th:4.1}, we finally obtain the
asymptotic expansion of the minimizer $\bsi_\vep^\star$
as well as of the Lagrangian multipliers:
\[
\bsi_\vep^\star=\bsi_0+\vep\bsi_1+\vep^2\bsi_2\text{satisfying}
\para\bsi_j\para_{\cH_\vep}\le c,\ j=0,1,2,
\]
and $\blam=\blam_0+\vep\blam_1+\vep^2\blam_2$
where
\[
(g_k,\si_\vep^k)_{L_2(\Gam_k)}=a^k,\ k\in I.
\]

It turns out that $\si_0^k$ and $\lam_0^k$ are exactly
the same as in \eqref{eq:4.5}. Moreover, we have to
replace $V_\bet$ by $(V_\bet-\bP\pp_\de)$ in the equations
\eqref{eq:4.8a} and \eqref{eq:4.8b}, \eqref{eq:4.8c}. Note that
$\bsi_0\in H^{\f{3}{2}\bet}(\Gam)\hookrightarrow H^{\ha\bet}(\Gam)$
and, hence, $\bP\pp_\de \si_0^k$ for $\de\to 0$ tends to zero in
$H^{\bet/2}(\Gam)$ due to \eqref{eq:4.14}.
(If $\bsi_0,\bsi_1\in C^1(\Gam)$ for $0<\bet<1$ or $\bsi_0,\bsi_1\in
C^2(\Gam)$ for $1\le\bet\le 2$, then $\bP\pp_\de\si_0^k$ and
$\bP\pp_\de\si_1^k=\cO (\de^{1-\bet})$, respectively $\cO
(\de^{2-\bet})$.)

Collecting these results, we have for the punched energy Riesz
minimum problem the following result:

\begin{theorem}\label{th:4.5}
Under the same assumptions as for Theorem \ref{th:4.1},
the minimization problem
\begin{equation}\label{eq:4.51}
\f{1}{2}\iint\limits_{\overset{0<\de\le|\bx-\by|}{\bx,\by\in\Gam}}
|\bx-\by|^{\al-n}\bsi(\by)\bsi(\bx)\d s_\by \d s_\bx-
\intli_\Gam \bof(\bx)\bsi(\bx)\d s_\bx\to\min,
\end{equation}
subject to \eqref{eq:4.2}, has for every $\de>0$ a unique
solution $\bsi_\vep^\star$. It admits for $\vep=\bet c_m\de^\bet>0$
the asymptotic expansion
\begin{equation}\label{eq:expansion}
\bsi_\vep^\star=\bsi_0+\vep\bsi_1+\vep^2\bsi_2\text{satisfying}
\para\bsi_j\para_{\cH_\vep}\le c,\ j=0,1,2
\end{equation}
with a constant $c$ independent of $\vep$. In particular, with
$\bsi_\vep^\star=(\al_k\si_0^k+\al_k\si_1^k+\al_k\si_2^k)_{k\in I}$
there holds
\begin{align*}
\si_0^k
&=
g_ka^k(g_k,g_k)_{L_2(\Gam_k)}^{-1},\\
\si_1^k
&=
\lam_1^k g_k-(V_\bet-\bP\pp_\de)\si_0^k+f_k,
\end{align*}
where
\[
\lam_1^k= (g_k,g_k)^{-1}
\big((V_\bet-\bP\pp_\de)\si_0^k-f_k,g_k\big)_{L_2(\Gam_k)}
\]
and
\[
\si_2^k =
\big(I+\vep(V_\bet-P_j\pp)\big)^{-1}
\big\{\lam_2^k g_k-(V_\bet-P_j\pp)\si_1^k\big\}
\]
with
\[
\lam_2^k= (g_k,g_k)^{-1}_{L_2(\Gam_k)}
\big\{(V_\bet-P_j\pp)\si_1^k+\vep(V_\bet-P_j\pp)\si_2^k+\si_2^k\big\}.
\]
\end{theorem}

\begin{cor}\label{c:4.6}
For $\de\to 0$ and $\vep=\bet c_m\de^\bet$ one finds
\[
\|\bsi_\vep^\star-\bsi_0\|_{L_2(\Gam)}\le c\vep
\overset{\de\to0}{\longrightarrow} 0
\]
with some constant $c$, independent of $\bsi_0$ and
$\vep>0$, where $\bsi_\vep^\star$ is the minimizer
\eqref{eq:4.51}. Moreover, it holds
\begin{equation}\label{eq:explosion}
\J_\de (\bsi_\vep^\star) =
\big(\V_\bet(\bsi_\vep^\star),\bsi_\vep^\star\big)_{L_2(\Gam)}-P_\de\p(\bsi_\vep^\star)
+\f{1}{\bet}\,\f{1}{c_m}\de^{-\bet}\|\bsi_\vep^\star\|^2_{L_2(\Gam)}
\overset{\delta\to 0}{\longrightarrow}\infty.
\end{equation}
\end{cor}

\begin{proof}
Due to \eqref{eq:expansion}, we obtain
\[
\|\bsi_\vep^\star-\bsi_0\|_{L_2(\Gam)}
\le
\para\bsi_\vep^\star-\bsi_0^\star\para_{\cH_\vep}\le
\vep\para\bsi_1\para_{\cH_\vep}
+\vep^2\para\bsi_2\para_{\cH_\vep}
\le
c(\vep+\vep^2).
\]
The conditions \eqref{eq:4kreuz} imply that $\si_0^k> 0$ and, hence,
$\|\bsi_\vep^\star\|_{L_2(\Gam)}\ge\f{1}{2}\|\bsi_0\|_{L_2(\Gam)}>0$ for all
$\vep$ with $0<\vep\le\vep_1$ with some $\vep_1> 0$. Thus,
there holds
\begin{align*}
\J_\de (\bsi_\vep^\star)
&=
\big(\V_\bet(\bsi_\vep^\star),\bsi_\vep^\star\big)_{L_2(\Gam)}-P_\de\p(\bsi_\vep^\star)
+\f{1}{\bet}\,\f{1}{c_m}\de^{-\bet}\|\bsi_\vep^\star\|^2_{L_2(\Gam)}\\
&\ge
\f{1}{2\bet}\,\f{1}{c_m}\de^{-\bet}\|\bsi_0\|^2_{L_2(\Gam)}-c\para\bsi_\vep^\star\para^2_{\cH_\vep}
\end{align*}
with uniformly bounded $\para\bsi_\vep^\star\para^2_{\cH_\vep}$. Hence,
$\de\to 0$ implies \eqref{eq:explosion}.
\end{proof}

\begin{rem}
For the torus $\Gam_1$ in $\R^3$, considered in \cite{dr}
and \cite{h-s}, where $\bof\in H^{\bet/2}(\Gam)$, $a^1>0$,
$g_1=1$, the minimizers $\bsi_\vep^\star$ of the punched
minimization problem tend to the constant charge $\bsi_\vep^\star\to
\bsi_0:=\f{1}{|\Gam|} a^1$.
\end{rem}

\appendix
\section{Explicit calculation of particular partie finie integrals}\label{a:a}
\setcounter{equation}{0}
\def\theequation{A.\arabic{equation}}
\setcounter{theorem}{0}
\def\thetheorem{A.\arabic{theorem}}
In this appendix, we shall compute the partie finie integrals
which define the functions $h(\bx)$ and $\bh(\bx)$ from
\eqref{eq:h} and \eqref{eq:bold h}, respectively.

\begin{lem}\label{l:a.1}
(i) Let $-1<\al<1$, $\Gam\in C^\8$ and $\bvfi\in C^\8(\Gam)$.
Then, one has
\[
\begin{gathered}
\lim_{\de\to 0} \pf \intli_{|\bx-\by|>\de>0}\ |\bx-\by|^{-\bet-(n-1)}
\big\{\bvfi(\by)-\bvfi(\bx)\big\} \d s_\by\\
=\pv \intli_{\Gam\sm\{\bx\}}|\bx-\by|^{-\bet-(n-1)}
\big\{\bvfi(\by)-\bvfi(\bx)\big\} \d s_\by,
\end{gathered}
\]
where $\pv$ denotes the Mikhlin--Calderon principal value integral.

(ii) The function $h(\bx)$ from \eqref{eq:h} is given by
\begin{equation}\label{eq:A.2}
\begin{gathered}
h(\bx)= \intli_{\by\in\Gam\wedge|\bx-\by|\le c}\
\big\{|\bx-\by|^{-\bet-(n-1)}-r^{-\bet-(n-1)}\big\}
\d s_\by-(\bet  c_n)^{-1}c^{-\bet}\\
-\pv \intli_{\by\in\Gam\wedge|\bx-\by|\le c}\ r^{-\bet-(n-1)}
\big\{r^{n-2}\d r\d\bom-\d s_\by\big\}\\
+\intli_{\by\in\Gam\wedge|\bx-\by|\ge c} |\bx-\by|^{-\bet-(n-1)}\d s_\by
\end{gathered}
\end{equation}
with any $c>0$ sufficiently small.

(iii) For the function $\bh(\bx)$ from \eqref{eq:bold h},
one obtains
\begin{equation}\label{eq:A.3}
\begin{gathered}
\bh(\bx)= \pv\intli_{\by\in\Gam\wedge 0<|\bx-\by|\le c}\
\big\{|\bx-\by|^{-\bet-(n-1)}(\by-\bx)
-r^{-\bet-(n-1)}(\by-\bx) \big\}\d s_\by\\
-\pv\intli_{\by\in\Gam\wedge |\bx-\by|\le c}\ r^{-\bet-(n-1)}
\big\{(\by-\bx)r^{n-2}\d r\d\bom-(\by-\bx)\d s_\by\big\}\\
+ \intli_{\by\in\Gam\wedge 0<|\bx-\by|\ge c}\
|\bx-\by|^{-\bet-(n-1)}(\by-\bx)\d s_\by
\end{gathered}
\end{equation}
with any $c>0$ sufficiently small.
\end{lem}

\begin{proof}
\emph{(i)\/}
Locally on $\Gam$ one has near $\bx\in\Gam$:
\begin{align*}
\bvfi(\by)
&=
\bvfi(\bx)+(\by-\bx)\cd\na \bvfi(\bx)+\cO(r^2),\
  r=|\bx-\by|;\\
\mathbf{\Theta}(r,\bom)
&:=
 \f{1}{r}(\by-\bx) \text{for}\bx, \by\in\Gam\secol \\
\bTheta(r,\bom)
&=
\bTheta(0,\bom)+\cO(r),\
\intli_{|\Theta|=1}\mathbf{\Theta}(0,\bom)\d\bom=0,\\
 \d s_\by
&=
r^{n-2}(1+\cO(r^2))\d r \d\bom.
\end{align*}
$|\bTheta (0,\bom)|=1$ describes the $(n-2)$-dimensional unit
sphere
$\bbS^{n-2}$
and $\bom$ is its polar coordinate with $\d\bom$
its $(n-2)$-dimensional surface measure.
Consequently, with an appropriate constant $c>0$ depending
on $\Gam$, one has
\begin{align*}
&\lim\limits_{\de\to 0}\pf\intli_{\by\in\Gam\wedge 0<\de<|\bx-\by|}
  |\bx-\by|^{-\bet-(n-1)}
\big\{\bvfi(\by)-\bvfi(\bx)\big\}\d s_\by\\
&\quad=\lim\limits_{\de\to 0+}\pf\intli_{\by\in\Gam\wedge 0<\de<|\bx-\by|\le
c}\
 |\bx-\by|^{-\bet-(n-1)}\\
 &\hspace*{3cm}\times\Big\{(\by-\bx)\cd\na\bvfi(\bx)
+\big\{\bvfi(\by)-\bvfi(\bx)-(\by-\bx)\cd\na\bvfi(\bx)\big\}\Big\}\d s_\by\\
&\qquad+\intli_{\by\in\Gam\wedge|\bx-\by|\ge c}\ |\bx-\by|^{-\bet-(n-1)}
\big\{\bvfi(\by)-\bvfi(\bx)\big\}\d s_\by\\
&\quad=\lim\limits_{\de\to 0+}\pf\intli_{0<\de\le r\le c}\
r^{-\bet}\d r
\intli_{|\bTheta|=1}\bTheta\d\bom(\bTheta)\cd\na\bvfi(\bx)\\
&\qquad+\pv\bigg\{\intli_{0<|\bx-\by|\le c} \big\{|\bx-\by|^{-\bet-(n-1)}
(\by-\bx) \d s_\by -r^{-\bet-1}(\by-\bx)\d r\d\bom\big\}\\
&\qquad+\intli_{0<|\bx-\by|\le c}   |\bx-\by|^{-\bet-(n-1)}
\big\{\bvfi(\by)-\bvfi(\bx)-(\by-\bx)\na\bvfi(\bx)\cd \d s_\by\big\}\bigg\}\\
&\qquad+\intli_{\by\in\Gam\wedge|\bx-\by|\ge c}|\bx-\by|^{-\bet-(n-1)}
\big\{\bvfi(\by)-\bvfi(\bx)\big\} \d s_\by.
\end{align*}
The first integral on the right is zero because of
$\int_{|\bTheta|=1}\bTheta\d\bom=\bnull$. Whereas, the integrand
of the remaining integral has at $\by=\bx$ a weak singularity
$|\by-\bx|^{-\bet+1}$ with $-1<-\bet+1<1$ whose principal value
integral exists. This follows by using e.g.\ Martensen's surface
polar coordinates, cf.~Theorem \ref{t:b1} in Appendix~\ref{a:b}.

\bigskip
\emph{(ii)\/} For the function
\[
h(\bx)
=\lim_{\de\to 0}\pf \intli_{0<\de<|\bx-\by|\wedge \by\in\Gam}
|\bx-\by|^{-\bet-(n-1)}\d s_\by,
\]
we find
\begin{align*}
h(\bx) &=\lim_{\de\to 0}\pf  \intli_{0<\de<|\bx-\by|\le c}
\big\{|\bx-\by|^{-\bet-(n-1)}-r^{-\bet-(n-1)}\big\}\d s_\by\\
&\qquad+\lim_{\de\to 0}\pf \intli_{0<\de}^c r^{-\bet-(n-1)}
r^{n-2}\d r\intli_{|\bTheta|=1}\d\bom\\
&\qquad-\lim_{\de\to 0}\pf\intli_{0<\de}^c\ \intli_{|\bTheta|=1}
r^{-\bet-(n-1)}\big\{r^{n-2}\d r\d\bom-\d s_\by\big\}\\
&\qquad
+\intli_{c\le |\bx-\by|\wedge\by\in\Gam}
|\bx-\by|^{-\bet-(n-1)}\d s_\by.
\end{align*}
If we choose Martensen's surface polar coordinates on $\Gam$,
then the integrand of the first integral on the right is identical zero
due to $|\bx-\by|=r$. For the second integral on the right, we have
with $\int_{|\bTheta|=1}\d\bom=c^{-1}_n$ that
\[
\lim_{0<\de\to 0}
\pf\intli_\de^c r^{\bet-1}\d r \intli_{|\bTheta|=1}\d\bom =
\lim_{0<\de\to 0}
\pf \f{1}{c_n}\bigg\{\f{\de^{-\bet}}{\bet}-\f{1}{\bet}c^{-\bet}\bigg\}
=-(c_n\bet)^{-1} c^{-\bet}.
\]
For the third integral, Theorem \ref{t:3.3} in
Appendix~\ref{a:b} implies
\[
r^{n-2} \d r\d\bom -\d s_\by =\cO(r^n) \d r\d\bom.
\]
Hence, the integrand is of order $\cO(r^{-\bet+1})$ and
the integral exists as a principal value integral. Collecting
these properties proves \eqref{eq:A.2}.

\bigskip
\emph{(iii)\/}
We proceed with respect to the vector-valued function
\[
\bh(\bx)=\pf\lim_{\de\to 0}\intli_{\by\in\Gam\wedge 0<\de<|\bx-\by|}
|\bx-\by|^{-\bet-(n-1)}(\by-\bx)\d s_\by,
\]
in the same manner as for $h(\bx)$. Inserting
$(\by-\bx)=r\bTheta$ for Martensen's surface
polar coordinates, we have:
\begin{align*}
\bh(\bx)
=&
\lim_{\de\to 0} \pf \intli_{0<\de<|\bx-\by|\le c} \big\{
|\bx-\by|^{-\bet-(n-1)}(\by-\bx)-r^{-\bet-(n-1)}(\by-\bx)\big\}\d s_\by\\
\quad&+
\lim_{\de\to 0}\pf\intli_{0<\de}^c r^{-\bet-(n-1)}r^{n-1}
\d r\intli_{|\bTheta|=1}\bTheta\d\bom\\
\quad&-
\lim_{\de\to 0}\pf\intli_\de^c\ \intli_{|\bTheta|=1} r^{-\bet-n}
\bTheta\big\{r^{n-2}\d r\d\bom-\d s_\by\big\}\\
\quad&+
\intli_{0\le |\bx-\by|\wedge\by\in\Gam}\
|\bx-\by|^{-\bet-(n-1)}(\by-\bx)\d s_\by.
\end{align*}
Here, the first integral on the right vanishes if Martensen's surface
polar coordinates are used, and the second one vanishes because of
$\int_{|\bTheta|=1}\bTheta \d\bom=\bnull$. The third integral contains an
integrand of order $\cO(r^{\bet+2})$. Hence, the integrand is bounded
and the integral exists. Consequently, also \eqref{eq:A.3}
holds.
\end{proof}

\section{Martensen's surface polar
  coordinates}\label{a:b}
\setcounter{equation}{0}
\def\theequation{B.\arabic{equation}}
\setcounter{theorem}{0}
\def\thetheorem{B.\arabic{theorem}}
For $n=3$, surface polar coordinates have been introduced by Martensen
in \cite[Chapter 2.1]{Ma}. A graphical illustration of these coordinates are
found in Figure~\ref{fig:sphere}. We generalize this approach to arbitrary
spatial dimension $n\ge 2$.

\begin{figure}[hbt]
\begin{center}
\includegraphics[height=4.2cm,width=6cm]{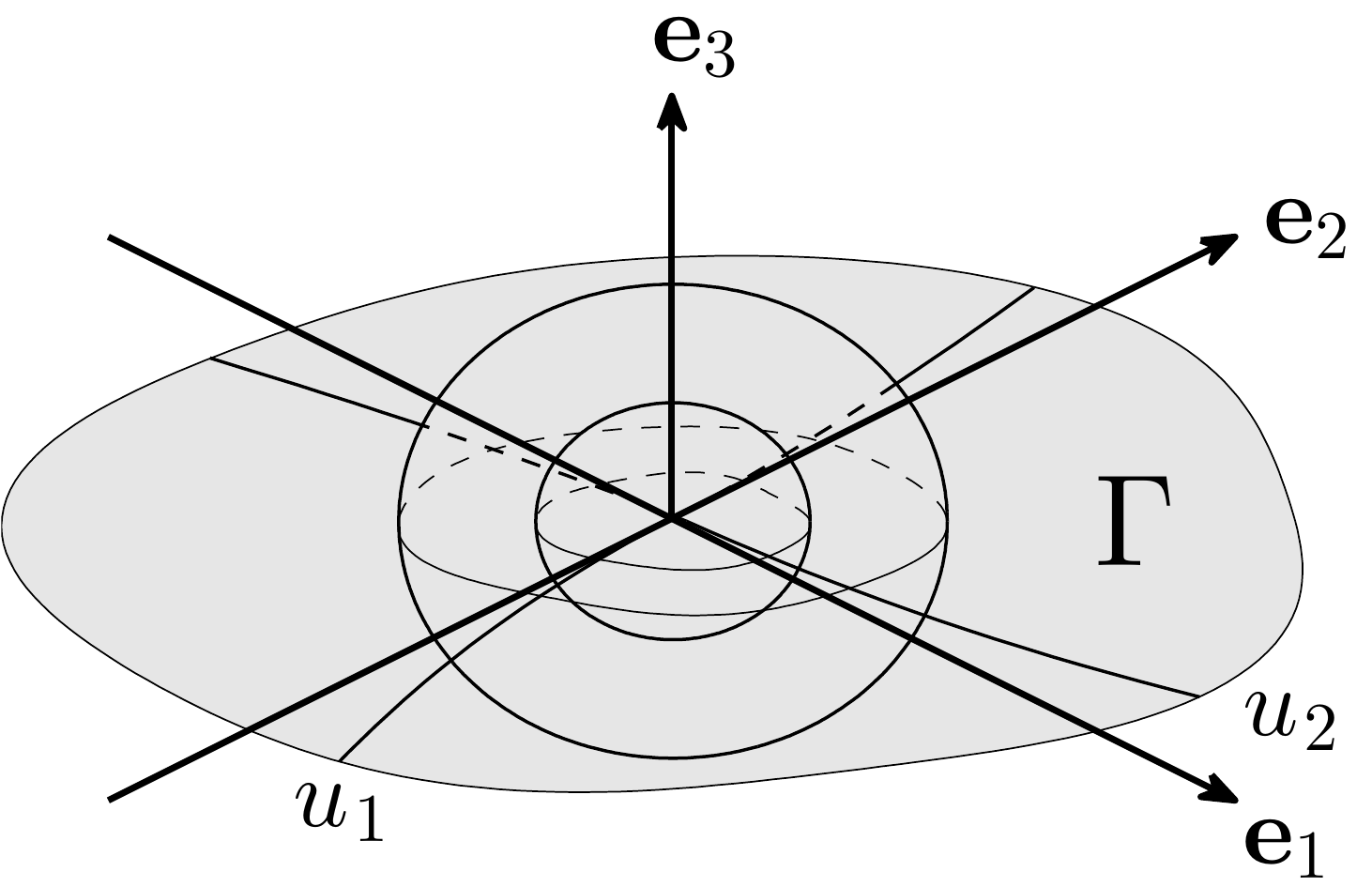}
\caption{\label{fig:sphere}Illustration of Martensen's surface polar coordinates.}
\end{center}
\end{figure}

The given surface has locally the parametric representation
$\Gam\dopu\bx=\bx(\bu)\in\R^n$,  $\bu=(u^1,\dots,u^{n-1})\in\R^{n-1}$.
For $\bx=\overset{\circ}{\bx} =\bx(\overset{\circ}{\bu})$, define the family
of $(n-2)$-dimensional manifolds (level sets) as
$\{\bx\in\cC_\vr\sbs\Gam$ given by $|\bx-\overset{\circ}{\bx}|=\vr>0\}$,
where $|\bx-\overset{\circ}{\bx}|$ denotes the Euclidian distance in
$\R^n$ and $\vr$ is the radial parameter.
Let
\begin{align*}
A(\bu)
&:=
|\bx(\bu)-\overset{\circ}{\bx}|,\\[-.5cm]
\intertext{then}
A^2(\bu)
&=
\big(\bx(\bu)-\overset{\circ}{\bx}\big)\cd\big(\bx(\bu)-\overset{\circ}{\bx}\big),
\end{align*}
where $\cd$ denotes the Euclidian scalar product in $\R^n$.
In the local neighborhood of $\Gam$, define
\[
\bx(\bu)=\overset{\circ}{\bx}+\bF(\bu)+G(\bu)\bn(\bu)
\]
with $G(\bu)=\big(\bx(\bu)-\overset{\circ}{\bx} \big)\cd \bn(\bu)$, where
$\bn(\bu)$ denotes the (exterior) unit normal vector of $\Gam$ at
$\bx=\bx(\bu)$. Then, on $\Gam$, the vector $\bF(\bu)$ is
tangential to $\Gam$ at $\bx(\bu)$,
\begin{equation}\label{eq:b1}
\bF\cd\bn=0,\ \bF
=
\big(\bx(\bu)-\overset{\circ}{\bx}\big)-G(\bu)\bn(\bu)\ \text{and}\
A^2
=
\bF\cd \bF+G^2.
\end{equation}

For $\overset{\circ}{\bx}\in\Gam$ chosen, the surface polar coordinates
then locally are given by the family of closed surfaces $\cC_\vr$, i.e.,
the level sets of the function $A(\bu)$, and curved radial rays
on $\Gam$ through $\overset{\circ}{\bx}$ which are perpendicular to
$\cC_\vr$ for constant $\vr$. Let us denote such a radial ray curve by
$\bc_{\bTheta}\in\Gam$, given by $\bu=\bu(s)$, $0\le s$, where the
parametric representation $\bx\big(\bu(s)\big)$ at $s=0$ starts in
$\overset{\circ}{\bx} =\bx(\overset{\circ}{\bu})$ in the direction of
the unit vector $\be(\bTheta)=\bx_{|i}(\overset{\circ}{\bu})\Theta^i$,
$\bx_{|i}=\f{\pa\bx}{\pa u^i}$, $i=1,\dots,n-1\secol$
$g_{jk}=\bx_{|j}\cd \bx_{|k}$, $g_{jk}\Theta^j\Theta^k=1$. Hence,
for the curve $\bc_{\bTheta}\dopu \bx=\bx\big(\bu(s)\big)$, we require that
\[
\bx_{|j}(\overset{\circ}{\bu})\f{\d u^j}{\d s}(\overset{\circ}{\bu})=\be(\bTheta).
\]
(Note that for $n=3$ one usually uses $\Theta_1=\cos \zeta$ and
$\Theta_2=\sin\zeta$.)

Without restriction of generality, we assume in what follows that at
$\os{\circ}{\bx}$ we have $g_{jk}(\os{\circ}{\bu})=\de_{jk}$ (the
Kronecker tensor).
Since the radial ray curves $\bc_{\bTheta}\in\Gam$ are perpendicular to
the level sets, which implicitly are given by $A=\vr=$ const, they satisfy the
ordinary differential equations for fixed $\bTheta$:
\begin{equation}\label{eq:b2}
\f{\pa\bc_{\bTheta}}{\pa s}=\f{\d\bx}{\d s}\Big|_{\bc_{\bTheta}}=
\bx_{|j} \f{\d u^j}{\d s}= \f{\Gr A}{|\Gr A|^2}(\bu)
\end{equation}
where $\Gr A= g^{\ell k}A_{|\ell}\bx_{|k}$
is the surface gradient and $g^{\ell k}g_{jk}=\de_{j\ell}$.
Since \eqref{eq:b1} implies on $\Gam$ that
\[
AA_{|i}=
(\bx-\overset{\circ}{\bx})\cd \bx_{|i}=F_i
=\bF\cd\bx_{|i}\text{and}
\Gr A
=
g^{\ell k}\f{F_\ell}{A} \bx_{|k} =\f{1}{A} F^k \bx_{|k}=\f{1}{A}\bF,
\]
we find
\[
\Gr A\cd \Gr A= \f{1}{A^2} \bF\cd\bF= \f{1}{A^2}(A^2-G^2)
=1-\f{G^2}{A^2}
\]
and the differential equations \eqref{eq:b2} take the form
\begin{equation}\label{eq:b3a}
\f{\d\bx}{\d s}\Big|_{\bc_{\bTheta}}=\bx_{|j} \f{\d u^j}{\d s}=
\f{\Gr A}{|\Gr A|^2} \Big(1-\f{G^2}{A^2}(\bu)\Big)^{-1}.
\end{equation}
After multiplication with $\bx_{|k}g^{\ell k}$, we thus arrive at
\begin{equation}\label{eq:b4}
\f{\d u}{\d s}^\ell =\f{1}{A(\bu)} \Big(1-\f{G^2}{A^2}(\bu)\Big)^{-1}
F^\ell(\bu),\  F^\ell=g^{\ell k}\bF\cd\bx_{|k},\ \ell,k=1,\dots,n-1.
\end{equation}

\begin{theorem}\label{t:b1}
To every $\overset{\circ}{\bx}\in\Gam$ there exists a
neighborhood of $\overset{\circ}{\bx}$ on $\Gam$ where
the system \eqref{eq:b4} admits a unique solution
\[
u^\ell (s,\bTheta)=s\Theta^\ell +\cO(s^2)
\]
for $s\ge 0$ which is the solution of the Volterra integral
equations
\begin{equation}\label{eq:b5}
u^\ell (s,\bTheta)=\overset{\circ}{u}^\ell+ s\Theta^\ell+
\intli_0^s
\bigg\{\f{F^\ell\big(\bu(\si,\bTheta)\big)}{A\big(\bu(\si,\bTheta)\big)}
\Big(1-\f{G^2\big(\bu(\si,\bTheta)\big)}{A^2\big(\bu(\si,\bTheta)\big)}\Big)^{-1}
-\Theta^\ell\bigg\}\d\si.
\end{equation}
The transformation $\bu\mapsto(\vr,\bTheta)$ to surface polar coordinates
about $\overset{\circ}{\bx}$ is given by
$\bu(\vr,\bTheta)$, i.e., $s=\vr\ge 0$.
\end{theorem}

\begin{proof}
Since
\[
\bF=
(\bx-\overset{\circ}{\bx})-\big((\bx-\overset{\circ}{\bx})\cd\bn\big)\bn,
\]
the expansion about $\overset{\circ}{\bx}
=\bx(\overset{\circ}{\bu})$ gives on the one hand
\begin{align}\label{eq:B6a}
\bx-\overset{\circ}{\bx}&=
\overset{\circ}{\bx}_{|k} \f{\d\overset{\circ}{u}}{\d s}^k s+\ha
\bigg\{\overset{\circ}{\bx}_{|k|\ell}
\f{\d\os{\circ}{u}}{\d s}^k \f{\d\os{\circ}{u}}{\d s}^\ell
+\overset{\circ}{\bx}_{|k}\f{\d^2\os{\circ}{u}}{\d s^2}^k\bigg\}s^2\no\\
&\quad+\f{1}{6}\bigg\{ \overset{\circ}{\bx}_{|k|\ell|j}\f{\d\overset{\circ}{u}}{\d s}^k\
\f{\d\overset{\circ}{u}}{\d s}^\ell\ \f{\d\overset{\circ}{u}}{\d s}^j
+3 \overset{\circ}{\bx}_{|k|\ell}\f{\d^2\overset{\circ}{u}}{\d s^2}^k
\ \f{\d\overset{\circ}{u}}{\d s}^\ell+ \bx_{|k}\f{\d^3\overset{\circ}{u}}{\d s^3}^k\bigg\}s^3
+\cO(s^4).
\end{align}
On the other hand, with the Gaussian equations
\[
\bx_{|j|k}
=\Gam_{jk}^\ell \bx_{|\ell}+L_{jk}\bn,\ \ell=1,\dots,n-1
\]
and the Weingarten relations
\[
\bn_{|j}=-L_j^m\bx_{|m},
\]
we obtain
\[
\bx_{|j|k|\ell}
=
\{\Gam_{jk|\ell}^m+\Gam_{jk}^r\Gam_{r\ell}^m-L_{jk} L_\ell^m\}
\bx_{|m}+\{L_{jk|\ell}+\Gam_{jk}^rL_{r\ell}\}\bn
\]
where $\Gam_{tj}^\ell$ are the Christoffel symbols of the second kind
of $\Gam$ at $\bx$, $\Gam_{tj|k}^\ell=\f{\pa}{\pa
  u^k}\Gam_{tj}^\ell(\bu)$, and $L_j^\ell=g^{\ell k} L_{kj}$ with
$L_{kj}$ the second fundamental form of $\Gam$ at $\bx$
(see e.g.\ \cite[p.\ 90]{ku}).
Here and in what follows, we abbreviate
\[
\overset{\circ}{\bx}_{|k}=\bx_{|k}(\overset{\circ}{\bu}),\
\overset{\circ}{\bx}_{|k|\ell}=
\bx_{|k|\ell} (\overset{\circ}{\bu})
\text{and} \f{\d\os{\circ}{u}}{\d s}^k= \f{\d u}{\d s}(\os{\circ}{u})\text{etc.}
\]
We get thus from \eqref{eq:B6a}
\begin{align}\label{eq:B7}
\bx-\overset{\circ}{\bx}&=\be(\bTheta)s+\ha
\bigg\{\Big(\overset{\circ}{\Gam}_{k\ell}^m\overset{\circ}{\bx}_{|m}
+\overset{\circ}{L}_{k\ell}\overset{\circ}{\bn}\Big)\Theta^k\Theta^\ell +
\overset{\circ}{\bx}_{|k}
\f{\d^2\overset{\circ}{u}^k}{\d s^2}\bigg\}s^2\no\\
&\quad+\f{1}{6}\bigg\{\Big(\Big[\os{\circ}{\Gam}_{jk|\ell}^t+\os{\circ}{\Gam}_{jk}^m
\os{\circ}{\Gam}_{m\ell}^t-\os{\circ}{L}_{jk}\os{\circ}{L}_\ell^t\Big]
\os{\circ}{\bx}_{|t}+
\Big[\os{\circ}{L}_{jk|\ell}+\os{\circ}{\Gam}_{jk}^m\os{\circ}{L}_{m\ell}\Big]
\os{\circ}{\bn}\Big)\Theta^j\Theta^k\Theta^\ell\\
&\quad\quad+3\Big(\os{\circ}{\Gam}_{k\ell}^t\os{\circ}{\bx}_{|t}+\os{\circ}{L}_{k\ell}
\os{\circ}{\bn}\Big)\Theta^\ell
\f{\d^2\os{\circ}{u}^k}{\d s^2}+\os{\circ}{\bx}_{|t}
\f{\d^3\os{\circ}{u}^t}{\d s^3}\bigg\} s^3+\cO(s^4).\no
\end{align}
By combining this expansion with
\[
\bn(\bu)=
\overset{\circ}{\bn}+\overset{\circ}{\bn}_{|j}
\f{\d\overset{\circ}{u}^j}{\d s}s+\ha
\bigg\{\os{\circ}{\bn}_{|j|k}\Theta^k\Theta^j+\os{\circ}{\bn}_{|j}
\f{\d^2\os{\circ}{u}^j}{\d s^2}\bigg\}s^2
 +\cO(s^3),
 \]
we arrive at
\begin{align*}
G
&=
(\bx-\overset{\circ}{\bx})\cd\bn(\bu)\\
&=-\overset{\circ}{\bx}_{|k} \cd \overset{\circ}{\bx}_{|m}
\Theta^k \overset{\circ}{L}^m_j \Theta^j s^2
+\ha \overset{\circ}{L}_{k\ell}\Theta^k\Theta^\ell s^2
-\ha L_{jk}\Theta^k\f{\d^2\os{\circ}{u}^j}{\d s^2}s^3\\
&\qquad+\bigg\{\f{1}{6}\Big(\os{\circ}{L}_{jk|\ell}
+\os{\circ}{\Gam}_{jk}^m\os{\circ}{L}_{m\ell}\Big)
-\ha\Big(\os{\circ}{L}_{jm}\os{\circ}{\Gam}_{k\ell}^m+
\os{\circ}{L}_{j|k}^\ell+\os{\circ}{L}_j^m\os{\circ}{\Gam}_{mk}^\ell\Big)\bigg\}
\Theta^j\Theta^k\Theta^\ell s^3+\cO(s^4)\\
&=-\ha  \overset{\circ}{L}_{kj}\Theta^k\Theta^j s^2
+\bigg\{\f{1}{6}\os{\circ}{L}_{jk|\ell}-\f{1}{3}\os{\circ}{L}_{jm}
\os{\circ}{\Gam}_{k\ell}^m-\ha\os{\circ}{L}_{j|k}^\ell-\ha\os{\circ}{L}_j^m
\os{\circ}{\Gam}_{mk}^\ell\bigg\}\Theta^j\Theta^k\Theta^\ell s^3\\
&\qquad-\ha L_{jk}\Theta^k\f{\d^2\os{\circ}{u}^j}{\d s^2}s^3+\cO(s^4),\\
\bF
&=
\be(\bTheta)s+\bigg\{\ha\Big(\Big[\overset{\circ}{\Gam}^m_{k\ell}\overset{\circ}{\bx}_{|m}
+\overset{\circ}{L}_{k\ell}\overset{\circ}{\bn}\Big]\Theta^k
\Theta^\ell+\overset{\circ}{\bx}_{|k}\f{\d^2\overset{\circ}{u}^k}{\d s^2}\Big)
+\ha \overset{\circ}{L}_{k\ell}\Theta^k\Theta^\ell\overset{\circ}{\bn}\bigg\}s^2+\cO(s^3),\\
A
&=
s+\cO(s^2).
\end{align*}
Hence, we conclude
\[
\f{\bF}{A}\Big(1-\f{G^2}{A^2}\Big)^{-1}=\be(\bTheta)+\cO(s)
\]
for any $C^2$-curve $\bx\big(\bu(s)\big)$ through
$\overset{\circ}{\bx}$. Consequently, the kernel function $\{\dots\}$
of the Volterra operator in \eqref{eq:b5} is continuous and there
exists a solution $\bu(s,\bTheta)$ for fixed given $\bTheta$ in some
vicinity of $\overset{\circ}{\bx}$ on $\Gam$.
This solution is in
$C^1([0,S])$ for some $S>0$ and is as many times continuously
differentiable for $s\ge 0$ as is the manifold $\Gam$.

The equation \eqref{eq:b3a} implies
\[
1=
\Gr A\cd \bx_{|j}\f{\d u}{\d s}^j = g^{\ell k}A_{|k}\bx_{|\ell}\cd\bx_{|j}
\f{\d u}{\d s}^j
=
g^{\ell k}g_{\ell j}A_{|k}\f{\d u}{\d s}^j=A_{|j}\f{\d u}{\d s}^j
=\f{\d A}{\d s}=\f{\d\vr}{\d s}.
\]
Thus, it holds $A=s=\vr$ and $\bu(\vr,\bTheta)$, the solution of
\eqref{eq:b5}, is the desired transformation $(\vr,\bTheta)\mapsto \bu$.

By bootstrapping, it follows from \eqref{eq:b4} that $\bu(\vr,\bTheta)$
is higher order differentiable up to $\vr=0$. To see this, consider
the Taylor expansions of the left and the right hand sides of the
equations \eqref{eq:b4} about $\overset{\circ}{\bu}$ up to the
order two:

\bigskip
\noindent
\underline{Left hand side of \eqref{eq:b4}:} (by using \eqref{eq:b3a})
\begin{align*}
&\f{\d u}{\d\vr}^\ell \bx_{|\ell} (\bu)
=
\f{\d\os{\circ}{u}}{\d\vr}^\ell \os{\circ}{\bx}_{|\ell} +
\bigg\{\overset{\circ}{\bx}_{|\ell}\f{\d^2\overset{\circ}{u}}
{\d\vr^2}^\ell +\overset{\circ}{\bx}_{|\ell|j}\
\f{\d\overset{\circ}{u}}{\d\vr}^\ell\
\f{\d\overset{\circ}{u}}{\d\vr}^j\bigg\}\vr\\
&
\qquad+\ha\bigg\{\overset{\circ}{\bx}_{|\ell}\f{\d^3
\overset{\circ}{u}}{\d\vr^3}^\ell +\overset{\circ}{\bx}_{|\ell|j|k}
\f{\d\overset{\circ}{u}}{\d\vr}^\ell\ \f{\d\overset{\circ}{u}}
{\d\vr}^j \ \f{\d\overset{\circ}{u}}{\d\vr}^k
+3\overset{\circ}{\bx}_{\ell|j} \f{\d^2\overset{\circ}{u}}
{\d\vr^2}^\ell\ \f{\d\overset{\circ}{u}}{\d\vr}^j\bigg\}\vr^2
+\cO(\vr^3) \\
&\quad=
\overset{\circ}{\bx}_{|\ell}\Theta^\ell + \bigg\{\overset{\circ}
{\bx}_{|\ell}\f{\d^2\overset{\circ}{u}}{\d\vr^2}^\ell
+\Big(\overset{\circ}{\Gam}_{\ell j}^m \overset{\circ}{\bx}_{|m}
+\overset{\circ}{\Gam}_{\ell j}\overset{\circ}{\bn}\Big)\Theta^\ell\Theta^j\bigg\}\vr\\
&\qquad+\ha\bigg\{\overset{\circ}{\bx}_{|\ell}\f{\d^3\overset{\circ}{u}}{\d\vr^3}^\ell
+\Big(\Big[\overset{\circ}{\Gam}_{\ell j|k}^m+\overset{\circ}
{\Gam}_{\ell j}^t\overset{\circ}{\Gam}_{tk}^m-\overset{\circ}{L}_{\ell
  j}\overset{\circ}{L}_k^m\Big]\overset{\circ}{\bx}_{|m}
+\Big[\overset{\circ}{\Gam}_{\ell j}^m \overset{\circ}{L}_{mk}
+\overset{\circ}{L}_{\ell j|k}\Big]\overset{\circ}{\bn}\Big)
\Theta^\ell\Theta^j\Theta^k\\
&
\qquad+3\Big(\overset{\circ}{\Gam}_{\ell j}^m\overset{\circ}{\bx}_{|m}
+\overset{\circ}{L}_{\ell j}\overset{\circ}{\bn}\Big)
 \Theta^j\f{\d^2\overset{\circ}{u}}{\d\vr^2}^\ell\bigg\}\vr^2+\cO(\vr^3).
\end{align*}

\bigskip
\noindent
\underline{Right hand side of \eqref{eq:b4}:}  (see \eqref{eq:B7})
\begin{align*}
&\f{1}{\vr}
(\bx-\overset{\circ}{\bx})
=
\overset{\circ}{\bx}_{|\ell}\Theta^\ell +\ha
\bigg\{\Big(\overset{\circ}{\Gam}_{k\ell}^m\overset{\circ}{\bx}_{|m}
+\overset{\circ}{L}_{k\ell}\overset{\circ}{\bn}\Big)
\Theta^k\Theta^\ell+ \overset{\circ}{\bx}_{|k}
\f{\d^2\overset{\circ}{u}}{\d\vr^2}^k\bigg\}\vr\\
&\quad+\f{1}{6}\bigg\{\overset{\circ}{\bx}_{|\ell}
\f{\d^3\overset{\circ}{u}}{\d\vr^3}^\ell+\Big(
\Big[\overset{\circ}{\Gam}_{\ell j|k}^m
+\overset{\circ}{\Gam}_{\ell j}^t\overset{\circ}{\Gam}_{tk}^m
-\overset{\circ}{L}_{\ell
  j}\overset{\circ}{L}_k^m\Big]\overset{\circ}{\bx}_{|m}
+\Big[\overset{\circ}{\Gam}_{\ell j}^m \overset{\circ}{L}_{mk}
+\overset{\circ}{L}_{\ell j|k}\Big]\overset{\circ}{\bn}\Big)
\Theta^\ell\Theta^j\Theta^k\\
&\quad+3\Big(\overset{\circ}{\Gam}_{\ell j}^m\overset{\circ}{\bx}_{|m}
+\overset{\circ}{L}_{\ell j}\overset{\circ}{\bn}\Big)\Theta^j
\f{\d^2\overset{\circ}{u}}{\d\vr^2}^\ell \bigg\}\vr^2+\cO(\vr^3), \\
&\bn(\bu)
=
\overset{\circ}{\bn}+\overset{\circ}{\bn}_{|j}
\f{\d\overset{\circ}{u}}{\d\vr}^j\vr+\ha
\bigg\{\overset{\circ}{\bn}_{|j|k}\f{\d\overset{\circ}{u}^j}{\d\vr}
\f{\d\overset{\circ}{u}}{\d\vr}^k+\overset{\circ}{\bn}_{|j}
\f{\d^2\overset{\circ}{u}}{\d\vr^2}^j\bigg\}\vr^2  +\cO(\vr^3)\\
&\quad=
\overset{\circ}{\bn}-\overset{\circ}{L}_j^m
\overset{\circ}{\bx}_{|m}\Theta^j\vr\\
&\qquad-\ha\bigg\{\Big(\overset{\circ}{L}_{j|k}^m\overset{\circ}{\bx}_{|m}
+ \overset{\circ}{L}_j^\ell\overset{\circ}{\Gam}_{\ell k}^m
\overset{\circ}{\bx}_{|m}+\overset{\circ}{L}_j^m
\overset{\circ}{L}_{mk}\overset{\circ}{\bn}\Big) \Theta^j\Theta^k
+\overset{\circ}{L}_j^m  \f{\d^2\overset{\circ}{u}}{\d\vr^2}^j
\overset{\circ}{\bx}_{|m}\bigg\}\vr^2+\cO(\vr^3),\\
&\f{1}{\vr}G
=
\ha\overset{\circ}{L}_{k\ell}\Theta^k\Theta^\ell \vr
-\ha \overset{\circ}{L}_{j\ell}\f{\d^2
\overset{\circ}{u}}{\d\vr^2}^j\Theta^\ell\\
&\quad
+\bigg\{\f{1}{6} \overset{\circ}{L}_{\ell j|k}-\f{1}{3}
\overset{\circ}{\Gam}_{\ell j}^m \overset{\circ}{L}_{mk}
-\ha \overset{\circ}{g}_{m\ell} \overset{\circ}{L}_{j|k}^m-\ha
\overset{\circ}{g}_{m\ell}\overset{\circ}{L}_j^t\overset{\circ}
{\Gam}_{tk}^m\bigg\}\Theta^j\Theta^k\Theta^\ell\vr^2+\cO(\vr^3),\\
&\Big(\f{1}{\vr}G\Big)\bn(\bu)
=\Big(\f{1}{\vr}G\Big)\Big(\overset{\circ}{\bn}-
\overset{\circ}{L}_j^m\Theta^j\overset{\circ}{\bx}_{|m}\vr\Big)+\cO(\vr^3)\\
&\quad=
\overset{\circ}{\bn}\ha\overset{\circ}{L}_{k\ell}\Theta^k\Theta^\ell\vr
+\overset{\circ}{\bn}\vr^2
-\ha\overset{\circ}{L}_j^m
\overset{\circ}{L}_{k\ell}\overset{\circ}{\bx}_{|m}
\Theta^j\Theta^k\Theta^\ell\vr^2
-\ha\overset{\circ}{L}_{j\ell}
\f{\d^2\overset{\circ}{u}}{\d\vr^2}^j\Theta^\ell\overset{\circ}{\bn}\vr^2
+\cO(\vr^3),\\
&\f{G^2}{A^2}
=
\Big(\f{1}{\vr}G\Big)^2= \f{1}{4}
\big(\overset{\circ}{L}_{k\ell}\Theta^k\Theta^\ell\big)^2\vr^2+\cO(\vr^3), \\
&\Big(1-\f{G^2}{A^2}\Big)^{-1}
=
1+\f{1}{4}\big(\overset{\circ}{L}_{k\ell}\Theta^k\Theta^\ell\big)^2\vr^2+\cO(\vr^3).
\end{align*}

\bigskip
\noindent
Comparing the coefficients of $\bx_{|\ell}\vr$ gives:
\[
\text{\underline{lhs:}}\quad
\f{\d^2\overset{\circ}{u}}{\d\vr^2}^\ell+\overset{\circ}{\Gam}_{jk}^\ell
\Theta^j\Theta^k
\quad\text{and}\quad
\text{\underline{rhs:}}\quad
\ha \overset{\circ}{\Gam}_{jk}^\ell \Theta^j\Theta^k
+\ha\f{\d^2\overset{\circ}{u}}{\d\vr^2}^\ell.
\]
Consequently, there holds
\[
\f{\d^2 \overset{\circ}{u}}{\d\vr^2}^\ell(\overset{\circ}{\bu})=
-\overset{\circ}{\Gam}_{jk}^\ell\Theta^j\Theta^k.
\]

\bigskip
\noindent
{\bf Next, compare the coefficients of $\overset{\circ}{\bx}_{|\ell}\vr^2$:}

\bigskip
\noindent
\underline{Left hand side of \eqref{eq:b4}:}
\begin{align*}
&\ha\ \f{\d^3\overset{\circ}{u}}{\d\vr^3}^\ell +\ha\Big(
  \overset{\circ}{\Gam}_{mj|k}^\ell+\overset{\circ}{\Gam}_{mj}^t
  \overset{\circ}{\Gam}_{tk}^\ell-\overset{\circ}{L}_{mj} \overset{\circ}{L}_k^\ell\Big)
\Theta^m\Theta^j\Theta^k
+\f{3}{2}\overset{\circ}{\Gam}_{mj}^\ell\Theta^j \f{\d^2\overset{\circ}{u}}{\d\vr^2}^m\\
&\quad=\ha\ \f{\d^3\overset{\circ}{u}}{\d\vr^3}^\ell+\ha\Big(\overset{\circ}
{\Gam}_{mj|k}^\ell+\overset{\circ}{\Gam}_{mj}^t\overset{\circ}
{\Gam}_{tk}^\ell-\overset{\circ}{L}_{mj} \overset{\circ}{L}_k^\ell\Big)
\Theta^m\Theta^j\Theta^k-\f{3}{2}\overset{\circ}{\Gam}_{tj}^\ell
\overset{\circ}{\Gam}_{km}^t\Theta^m\Theta^j\Theta^k.
\end{align*}

\bigskip
\noindent
\underline{Right hand side of \eqref{eq:b4}:}
\begin{align*}
&\Big(\f{1}{\vr} (\bx-\overset{\circ}{\bx})-\f{1}{\vr}G\bn\Big)
\Big(1+\f{1}{4}\vr^2(\overset{\circ}{L}_{k\ell}\Theta^k\Theta^\ell)^2
+\cO(\vr^3)\Big)\\
&\quad=
\Theta^\ell\os{\circ}{\bx}_{|\ell}-\vr\ha
\overset{\circ}{\Gam}_{km}^\ell \Theta^k\Theta^m
\overset{\circ}{\bx}_{|\ell}\\
&\qquad+
\bigg\{\f{1}{6}\ \f{\d^3\overset{\circ}{u}}{\d\vr^3}^\ell+\Big(\f{1}{6}
\overset{\circ}{\Gam}_{mj|k}^\ell-\f{1}{3}\overset{\circ}{\Gam}_{mj}^t
\overset{\circ}{\Gam}_{tk}^\ell
-\f{1}{3}\overset{\circ}{L}_{j}^\ell \overset{\circ}{L}_{km}\Big)
\Theta^m\Theta^j\Theta^k
+\f{1}{4}\big(\overset{\circ}{L}_{kj}\Theta^k\Theta^j\big)^2
\Theta^\ell\bigg\}\os{\circ}{\bx}_{|\ell}.
\end{align*}

\bigskip
\noindent
Hence, from equating left hand side and right side,
one obtains
\[
 \f{\d^3u}{\d\vr^3}^\ell(\overset{\circ}{\bu})=\Big\{-
\overset{\circ}{\Gam}_{mj|k}^\ell+2\overset{\circ}{\Gam}_{mj}^t
\overset{\circ}{\Gam}_{tk}^\ell+\ha \overset{\circ}{L}_{mj}\overset{\circ}{L}_k^\ell\Big\}
\Theta^m\Theta^j\Theta^k
+\f{3}{4}\big(\overset{\circ}{L}_{kj}\Theta^k\Theta^j\big)^2\Theta^\ell.
\]

With the expressions for $\f{\d^2\os{\circ}{u}^\ell}{\d s^2}$ and
$\f{\d^3\os{\circ}{u}^\ell}{\d s^3}$, we recollect the relations for
$\bx-\os{\circ}{\bx}$ and find instead of \eqref{eq:B6a}:
\begin{align*}
&(\bx-\bx_0)
=
\vr\be(\bTheta)+\ha\vr^2\os{\circ}{\bn}\os{\circ}{L}_{k\ell}
\Theta^k\Theta^\ell\\
&+
\f{\vr^3}{6}\bigg\{
\Big(-\ha
\os{\circ}{\bx}_{|\ell}\os{\circ}{L}_{mj}\os{\circ}{L}_k^\ell
+\Big[2\os{\circ}{L}_{mj|k}-3\os{\circ}{L}_{tj}\os{\circ}{\Gam}_{mk}^t\Big]
\os{\circ}{\bn}\Big)\Theta^m\Theta^j\Theta^k
+\f{4}{3}\big(\os{\circ}{L}_{kj}\Theta^k\Theta^j\big)^2\bTheta\bigg\}+\cO(\vr^4).
\end{align*}
Collecting the first three derivatives of $u^\ell$ at
$\overset{\circ}{\bu}$ implies that the first terms
of the transform $(\vr,\bTheta)\mapsto \bu$ read as
\begin{align}\label{eq:B8}
&u^\ell(\vr,\bTheta)
=
\overset{\circ}{u}^\ell+\Theta^\ell\vr-\ha\overset{\circ}{\Gam}_{jk}^\ell
\Theta^j\Theta^k\vr^2\\
&\qquad
+\f{1}{6}\bigg\{\ha \overset{\circ}{L}_j^\ell \overset{\circ}{L}_{km}
+2\overset{\circ}{\Gam}_{tj}^\ell\overset{\circ}{\Gam}_{km}^t
-\overset{\circ}{\Gam}_{mj|k}^\ell\bigg\} \Theta^j\Theta^k\Theta^m\vr^3
+\f{1}{8}\big(\overset{\circ}{L}_{jk}\Theta^j\Theta^k\big)^2\Theta^\ell\vr^3
+\cO(\vr^4),\no
\end{align}
for all $\ell=1,\dots,n-1$. Whereas, the inverse mapping
$\bu\mapsto(\vr,\bTheta)$ can be obtained  from
\[
\vr=A(\bu)=|\bx(\bu)-\overset{\circ}{\bx}|
\]
and the nonlinear equations for
\begin{align*}
\Theta^\ell
&=
\f{1}{\vr}\big(u^\ell-\overset{\circ}{u}^\ell\big)+\ha\vr
\overset{\circ}{\Gam}_{jk}^\ell\Theta^j\Theta^k\\
&\
-\f{1}{6}\vr^2\bigg\{\ha\overset{\circ}{L}_j^\ell \overset{\circ}{L}_
{km}+2\overset{\circ}{\Gam}_{tj}^\ell\overset{\circ}
{\Gam}_{km}^t-\overset{\circ}{\Gam}_{mj|k}^\ell\bigg\}
\Theta^j\Theta^k\Theta^m
-\f{1}{8}\vr^2\big(\overset{\circ}{L}_{jk}\Theta^j\Theta^k\big)^2
\Theta^\ell+\dots
\end{align*}
for all $\ell=1,\dots,n-1$ can, for $\vr>0$ sufficiently
small, be solved via successive iteration.
\end{proof}

\begin{theorem}\label{t:3.3}
Let $\Gam\in C^4$. Then, the surface measure $\d s_\Gam$
of $\Gam$ in Martensen's surface polar coordinates satisfies
\begin{equation}\label{eq:3++}
\d s_\Gam=\vr^{n-2}\d\vr\wedge\d\bom +\big(\vr^n a(\bTheta)+\cO(\vr^{n+1})\big)
\d\vr\wedge\d\bom
\end{equation}
where
\[
a(\bTheta)=\sumli_{j=1}^{n-1} \Theta^j \bigg\{\f{3}{4}\Big(
\big(\os{\circ}{L}_{kt}\Theta^k\Theta^t\big)^2+\os{\circ}{L}_t^j
\os{\circ}{L}_{mk}\Theta^t\Theta^m\Theta^k\Big)
-\Theta^j \os{\circ}{L}_m^\ell\os{\circ}{L}_{\ell
  k}\Theta^m\Theta^k\bigg\}
\]
and
\[
\d\bom=\sumli_{j=1}^{n-1}(-1)^{j+1}\Theta^j[\d\Theta^1\wedge\dots\wedge
\kreuz\wedge\dots\wedge\d\Theta^{n-1}],\
\sumli_{j=1}^{n-1}\Theta^j\Theta^j=1.
\]
Here, $\d\bom$ is the surface measure of the unit sphere $\bbS^{n-2}$ in
$\R^{n-1}$.
\end{theorem}

\begin{proof}
For the surface measure of $\Gam$, we use the exterior Pfaffian
products, defining the exterior normal vector's components times
$\d s_\Gam$ (see \cite[Chapter 11.4]{We-St}. Then, we multiply
scalarly with the exterior unit normal $\bn(\bx)= \big(n^1(\bx),
\dots,n^n(\bx)\big)$ which yields
\begin{equation}\label{eg:surface measure}
\d s_\Gam=\sumli_{j=1}^n (-1)^{j+1}\big[\d x^1\wedge\dots\wedge
  \kreuzx\wedge\dots\wedge\d x^n\big]n^j(\bx).
\end{equation}

Expansions about $\os{\circ}{\bx}$ up to the order $\vr^2$ give
\begin{align*}
\bn(\bx)=\os{\circ}{\bn}-\vr
\os{\circ}{L}_j^m\os{\circ}{\bx}_{|m}\Theta^j
&-\ha \vr^2 \bigg\{\os{\circ}{L}_{j|k}^m+\os{\circ}{L}_j^\ell
\overset{\circ}{\Gam}_{\ell k}^m-\overset{\circ}{\Gam}_{jk}^t
\os{\circ}{L}_t^m\bigg\}\Theta^j\Theta^k\os{\circ}{\bx}_{|m}\\
&\qquad-\ha\vr^2\os{\circ}{L}_j^\ell\os{\circ}{L}_{\ell k}\Theta^j\Theta^k
\os{\circ}{\bn}+\cO(\vr^3)
\end{align*}
and
\begin{align*}
\d x^j
&=
\d\vr\Theta^j +\vr\d\Theta^j +\ha \vr^2\d\vr
\bigg\{\f{3}{4} \big(\os{\circ}{L}_{kt}\Theta^k\Theta^t\big)^2
-\ha \os{\circ}{L}_k^j
\os{\circ}{L}_{mt}\Theta^k\Theta^m\Theta^t\bigg\}
+\cO(\vr^3)\\
&=
\d\vr\Theta^j +\vr\Theta^j +\vr^2\d\vr c^j(\bTheta)+\cO(\vr^3)
\end{align*}
with
\[
c^j(\bTheta)=\ha\bigg\{\f{3}{4}\big(\os{\circ}{L}_{kt}\Theta^k\Theta^t\big)^2
-\ha \os{\circ}{L}_k^j
\os{\circ}{L}_{mt}\Theta^k\Theta^m\Theta^t\bigg\}
\]
for $j=1,\dots,n-1$ and
\[
\d x^n=\os{\circ}{L}_{jk}\Theta^j\Theta^k\vr\d\vr+\ha\Big(
2\os{\circ}{L}_{mj|k}-3\os{\circ}{L}_{\ell
  j}\os{\circ}{\Gam}_{mk}^\ell\Big) \Theta^m\Theta^j\Theta^k\vr^2\d\vr+\cO(\vr^3).
\]
Inserting this into \eqref{eg:surface measure}, yields
\begin{align}\label{eq:x}
\d s_\Gam
&=
(-1)^{n+1}\big[\d x^1\wedge\dots\wedge\d x^{n-1}\big]
\Big(1-\ha\vr^2\os{\circ}{L}_j^m\os{\circ}{L}_{mk}\Theta^j\Theta^k
+\cO(\vr^3)\Big)\no\\
&\quad
+\sumli_{j=1}^{n-1}\Big[\big[\d x^1\wedge\dots\wedge\kreuzx\wedge\d x^{n-1}\big]
\wedge\d x^n\Big]\no\\
&\qquad
\x\bigg\{-\vr\os{\circ}{L}_k^j\Theta^k
-\ha\Big(\os{\circ}{L}_{m|k}^j+\os{\circ}{L}_m^\ell
\os{\circ}{\Gam}_{\ell k}^j-\os{\circ}{\Gam}_{mk}^t
\os{\circ}{L}_t^j\Big)\Theta^m\Theta^k\vr^2+\cO(\vr^3)\bigg\}.
\end{align}

For the first term in \eqref{eq:x}, we obtain
(modulo $\cO(\vr^3)$ terms) with the relations for
$\d x^j$:
\begin{align*}
&\big[\d x^1\wedge\dots\wedge\d x^{n-1}\big]
=
\big[\d\vr\Theta^1+\vr\d\Theta^1+\vr^2\d\vr c^1(\bTheta)
\wedge
\d\vr \Theta^2+\vr\d\Theta^2+\vr^2\d\vr c^2(\bTheta)\\
&\wedge\dots\wedge
\d\vr\Theta^{n-1}+\vr\d\Theta^{n-1}+\vr^2\d\vr c^{n-1}(\bTheta)\big]
=
\vr^{n-1}\big[\d\Theta^1\wedge\dots\wedge\d\Theta^{n-1}\big]\\
&+\vr^{n-2}\bigg[\d\vr\wedge\sumli_{j=1}^{n-1}(-1)^j
\big\{\Theta^j+\vr^2 c^j(\bTheta)\big\}
\big[\d\Theta^1\wedge\dots\wedge\kreuz\wedge\dots\wedge\d\Theta^{n-1}\big]
\bigg].
\end{align*}
Since the variables $\Theta^j$ vary on the $(n-2)$-dimensional
sphere $\bbS^{n-2}$, where
\[
\sumli_{j=1}^{n-1} (\Theta^j)^2=1,
\]
we have
\begin{equation}\label{eq:B11}
\sumli_{j=1}^{n-1}\Theta^j\d\Theta^j=0
\end{equation}
on $\bbS^{n-2}$. Hence, the differentials $\d\Theta^j$ with
$j=1,\dots,n-1$ are linearly dependent and
\[
\big[\d\Theta^1\wedge\dots\wedge\d\Theta^{n-1}\big]=0.
\]
Moreover, on $\bbS^{n-2}$, we have that the exterior
unit normal vector  $\bnu$ to  $\bbS^{n-2}$ satisfies
\[
\bnu_{\bom} = \sumli_{j=1}^{n-1} \Theta^j\be_j.
\]
Therefore, it holds
\[
\nu_j\d\bom=(-1)^{j+1}\big[\d\Theta^1\wedge\dots\wedge\kreuz\wedge
\dots\wedge\d\Theta^{n-1}\big]
=\Theta^j\d\bom,
\]
and with \eqref{eq:B11} one obtains
\[
\d\bom=\sumli_{j=1}^{n-1}(-1)^{j+1}\Theta^j
\big[\d\Theta^1\wedge\dots\wedge\kreuz\wedge\dots\wedge\d\Theta^{n-1}\big].
\]
For the first term in \eqref{eq:x}, we find therefore
\[
\d s_\Gam^1
=
(-1)^{n+1}\bigg\{
\vr^{n-2}+\vr^n\Big(\sumli_{j=1}^{n-1} c^j(\bTheta)\Theta^j\\
-\ha \os{\circ}{L}_m^\ell \os{\circ}{L}_{\ell k}\Theta^m\Theta^k \Big)
+\cO(\vr^{n+1})\bigg\}\d\vr\wedge\d\bom.
\]

For the remaining terms in \eqref{eq:x}, we have
(modulo $\cO(\vr^3)$ terms) that
\begin{align*}
\d s_\Gam^2
&=
\Big[\big[\d x^1\wedge\dots\wedge\kreuzx\wedge\dots\wedge\d x^{n-1}\big]\wedge
    \d x^n\Big]\Big(-\vr \os{\circ}{L}_t^j\Theta^t
    -\ha\vr^2\{\dots\}\Theta^m\Theta^k\Big)\\
&=\Big[\big[\d\vr\Theta^1+\vr\d\Theta^1+\vr^2\d\vr c^1(\bTheta)\wedge
\d\vr\Theta^2+\vr\d\Theta^2+\vr^2\d\vr c^2(\bTheta)\\
&\qquad\qquad\wedge\dots\wedge\kreuzx\wedge\dots\wedge\d\vr^2\Theta^{n-1}+
\vr\d\Theta^{n-1}+\vr^2\d\vr c^{n-1}(\bTheta)\big]\\
&\qquad\qquad\wedge\Big(\vr\d\vr\os{\circ}{L}_{jk}\Theta^j\Theta^k
+\ha \vr^2\d\vr \{\dots\}\Theta^m\Theta^j\Theta^k\Big)\Big]
\Big(-\vr \os{\circ}{L}_t^k \Theta^t-\ha\{\dots\}\Theta^m\Theta^k\Big)\\
&=
\sumli_{j=1}^{n-1}(-1)^{n+j}\big(\vr^n+\cO(\vr^{n+1})\big)\d\vr\wedge
\big[\d\Theta^1\wedge\dots\wedge\kreuz\wedge\dots\wedge\d\Theta^{n-1}\big]
\os{\circ}{L}_{mk}\Theta^m\Theta^k\os{\circ}{L}_t^j\Theta^t\\
&=
(-1)^{n+1}\sumli_{j=1}^{n-1}\os{\circ}{L}_{mk}\Theta^m\Theta^k
\os{\circ}{L}_t^j\Theta^t\Theta^j\big(\vr^n+\cO(\vr^{n+1})\big)
\d\vr\wedge\d\bom.
\end{align*}
Consequently, we finally get in \eqref{eq:x}
\[
\d s_\Gam=\d s_\Gam^1+\d s_\Gam^2=
\vr^{n-2}\d\vr\wedge\d\bom +\big(\vr^n a(\bTheta)+\cO(\vr^{n+1})\big)\d\vr\wedge
\d\bom
\]
with
\[
a(\bTheta)=\sumli_{j=1}^{n-1}\Theta^j \bigg\{\f{3}{4}
\big(\os{\circ}{L}_{kt}\Theta^k\Theta^t\big)^2+\os{\circ}{L}_t^j
\os{\circ}{L}_{mk}\Theta^t\Theta^m\Theta^k
-\Theta^j\os{\circ}{L}_m^\ell\os{\circ}{L}_{\ell k}\Theta^m\Theta^k
\bigg\}.
\]
This is the proposed relation \eqref{eq:3++} for the surface measure.
\end{proof}


\end{document}